\documentclass{scrartcl}
\usepackage{amsmath,amssymb,amsthm,enumerate,titlefoot, bm}
\usepackage[dvipdfmx]{graphicx}
\usepackage[all]{xy}
\usepackage{tabularx}

\newtheorem{thm}{Theorem}[section]
\newtheorem{prop}[thm]{Proposition}
\newtheorem{cor}[thm]{Corollary}
\newtheorem{lem}[thm]{Lemma}

\newtheorem{conj}[thm]{Conjecture}

\theoremstyle{definition}
\newtheorem{dfn}[thm]{Definition}
\newtheorem{ex}[thm]{Example}
\newtheorem{prop and dfn}[thm]{Proposition \& Definition}

\theoremstyle{remark}
\newtheorem{rem}[thm]{Remark}

\newcommand{\id}{\mathrm{id}}
\newcommand{\rank}{\mathrm{rank}}

\newcommand{\Ker}{\mathrm{Ker}}
\renewcommand{\Im}{{\mathrm{Im}}}

\newcommand{\pos}{\mathrm{pos}}

\newcommand{\bfit}[1]{\textbf{\textit{#1}}}
\newcommand{\simwrt}[1]{\underset{#1}{\sim}}
\newcommand{\nsimwrt}[1]{\underset{#1}{\not\sim}}

\newcommand{\inp}{{\mathrm{in}}_{\bfit p}}
\newcommand{\Txpm}{\mathbb T[\bfit x^{\pm}]}
\newcommand{\Txpmfcn}{\mathbb T[\bfit x^{\pm}]_{\mathrm{fcn}}}
\newcommand{\Bxpm}{\mathbb B[\bfit x^{\pm}]}
\newcommand{\Bypm}{\mathbb B[\bfit y^{\pm}]}
\newcommand{\Bxpmfcn}{\mathbb B[\bfit x^{\pm}]_{\mathrm {fcn}}}
\newcommand{\Bxnpmfcn}{\mathbb B[x_1^{\pm}, \ldots, x_n^{\pm}]_{\mathrm {fcn}}}
\newcommand{\Bypmfcn}{\mathbb B[\bfit y^{\pm}]_{\mathrm {fcn}}}
\newcommand{\Bympmfcn}{\mathbb B[y_1^{\pm}, \ldots, y_m^{\pm}]_{\mathrm {fcn}}}
\newcommand{\Cp}{\mathcal C_{\bfit p}}
\newcommand{\CpB}{\mathcal C_{\bfit p, \mathbb B}}
\newcommand{\CX}{\mathcal C(X)}
\newcommand{\CY}{\mathcal C(Y)}
\newcommand{\CXp}{\mathcal C(X)_{\bfit p}}
\newcommand{\CXo}{\mathcal C(X)_{\mathbf 0}}
\newcommand{\CXpB}{\mathcal C(X)_{\bfit p, \mathbb B}}
\newcommand{\CXoB}{\mathcal C(X)_{\mathbf 0, \mathbb B}}

\newcommand{\ZApos}{\mathbb Z^A_{\pos} \cup \{ -\infty \}}
\newcommand{\ZBpos}{\mathbb Z^B_{\pos} \cup \{ -\infty \}}
\newcommand{\ZXpos}{\mathbb Z^{X(1)}_{\pos} \cup \{ -\infty \}}
\newcommand{\ZYpos}{\mathbb Z^{Y(1)}_{\pos} \cup \{ -\infty \}}
\newcommand{\Zpos}[1]{\mathbb Z^{#1}_{\pos} \cup \{ -\infty \}}

\renewcommand{\epsilon}{\varepsilon}
\renewcommand{\phi}{\varphi}

\newcommand{\matab}[2]{\begin{pmatrix} #1 & #2 \end{pmatrix}}
\newcommand{\matba}[2]{\begin{pmatrix} #1 \\ #2 \end{pmatrix}}
\newcommand{\matbb}[4]{\begin{pmatrix} #1 & #2 \\ #3 & #4 \end{pmatrix}}

\newcommand{\matac}[3]{\begin{pmatrix} #1 & #2 & #3 \end{pmatrix}}

\newcommand{\matca}[3]{\begin{pmatrix} #1 \\ #2 \\ #3 \end{pmatrix}}

\newcommand{\matcc}[9]{\begin{pmatrix} #1 & #2 & #3 \\ #4 & #5 & #6 \\ #7 & #8 & #9 \end{pmatrix}}

\newcommand{\matad}[4]{\begin{pmatrix} #1 & #2 & #3 & #4 \end{pmatrix}}

\begin{document}

\title{\ \quad Local theory of functions on tropical curves in $\mathbb R^n$
\keywords{max-plus algebra, tropical algebra, tropical curves, local theory}
\amssubj{Primary 14T10, Secondary 15A80, 16Y60.}
}
\author{Takaaki Ito}
\date{\hfill}
\maketitle

\begin{table}[h]
\centering
\begin{tabularx}{0.7\linewidth}{X}
\begin{center} \textbf{Abstract} \end{center}\\
{\small
We first develop the local theory of functions on $\mathbb R^n$ defined by tropical Laurent polynomials.
We study the structure of the semiring of functions, where two functions are identified when they coincide on a neighborhood of a fixed point.
We see that this semiring is closely related to the semiring of functions defined by Boolean Laurent polynomials.
Then we develop the local theory of functions on tropical curves.
We construct a contravariant functor from the category of 1-dimensional tropical fans to the category of certain homomorphisms of semirings.
As an application, we discuss about the smoothness of 1-dimensional tropical fans at the origin.
}
\end{tabularx}
\end{table}

\section{Introduction}
Recently, relations between tropical varieties and tropical polynomials are studied in several ways.
In \cite{GG}, J. Giansiracusa and N. Giansiracusa construct tropical hypersurfaces as schemes over idempotent semirings.
Maclagan and Rinc\'{o}n define tropical ideals in \cite{MR}, which are ideals in the semiring of tropical polynomials, and show that tropical ideals have many properties similar to classical ones.
In \cite{BE}, Bertram and Easton study congruences defined by subsets of the tropical affine space $\mathbb T^n$ and subsets defined by congruences on the semiring of tropical polynomials.

In this paper, we give another approach.
We focus on a neighborhood of a point.
Let $\mathbb T [x_1^{\pm}, \ldots, x_n^{\pm}]_{\mathrm{fcn}} = \Txpmfcn$ be the semiring which consists of functions on $\mathbb R^n$ defined by tropical Laurent polynomials.
Let $X$ be a $1$-dimensional tropical fan in $\mathbb R^n$ and $|X|$ be its support.
Let us consider the congruence
$$\mathcal C(X)_{\mathbf 0} = \left\{ (f, g)  \ \middle| \ 
\begin{aligned}
&\text{ there exists an open neighborhood $U$ of $\mathbf 0$ } \\
&\text{ such that }  f|_{U \cap |X|} = g|_{U \cap |X|}
\end{aligned} \right\}$$
on $\Txpmfcn$, where the topology on $\mathbb R^n$ is the Euclidean topology.
We study the relation between $X$ and $\Txpmfcn / \CXo$.
Since tropical curves are locally isomorphic to tropical fans, this is the local theory of tropical curves.

We first develop the local theory of functions on $\mathbb R^n$ defined by tropical Laurent polynomials.
Let $\bfit p \in \mathbb R^n$ be any vector.
Consider the congruence
$$\mathcal C_{\bfit p} = \left\{ (f, g)  \ \middle| \ 
\begin{aligned}
&\text{ there exists an open neighborhood $U$ of $\bfit p$ } \\
&\text{ such that }  f|_U = g|_U
\end{aligned} \right\}$$
on $\Txpmfcn$.
The quotient semiring $\Txpmfcn / \mathcal C_{\bfit p}$ has a distinguished structure, that is, $\Txpmfcn / \mathcal C_{\bfit p}$ is isomorphic to the $\mathbb T$-extension (see Definition \ref{Textdef}) of $\Bxpmfcn$, where $\mathbb B = \{ -\infty, 0 \}$ is the Boolean semifield.
By using properties of $\mathbb T$-extensions, we obtain our first main theorem.

\begin{thm}
\label{introthm1}
\begin{enumerate}[$(1)$]
\item Any congruence $\mathcal C$ on $\Txpmfcn$ including $\mathcal C_{\bfit p}$ satisfies just one of the following conditions:
\begin{enumerate}[$(a)$]
\item if $(f, g) \in \mathcal C$, then $f(\bfit p) = g(\bfit p)$, and
\item $(f, g) \in \mathcal C$ for any $f, g \in \Txpmfcn \smallsetminus \{ -\infty \}$.
\end{enumerate}
\item If $\mathcal C$ is a congruence on $\Txpmfcn$ including $\mathcal C_{\bfit p}$ and satisfying the condition (a), then there exists a congruence $\mathcal C'$ on $\Bxpmfcn$ such that
$\Txpmfcn / \mathcal C $ is isomorphic to the $\mathbb T$-extension of $\Bxpmfcn / \mathcal C'$.
\end{enumerate}
\end{thm}

By applying Theorem \ref{introthm1} to $\CXo$, it is shown that there exists a congruence $\CXoB$ on $\Bxpmfcn$ such that the semiring $\Txpmfcn / \CXo$ is isomorphic to the $\mathbb T$-extension of $\Bxpmfcn / \CXoB$.
Moreover, we will show that $\CXoB$ coincides with
$$\CX := \left\{ (f, g) \in \left( \Bxpmfcn \right)^2 \ \middle| \ f|_{|X|} = g|_{|X|} \right\}.$$
One may expect that we can reconstruct $X$ from the semiring $\Bxpmfcn / \CX$.
However, it is impossible in general (see Remark \ref{rem:reconst}).
Instead of that, we construct a semiring homomorphism $\phi_X$ associated with $X$, and show that we can reconstruct $X$ from $\phi_X$.
The map $\phi_X$ is defined as follows:
For two sets $A,B$, we denote by $B^A$ the set of all maps from $A$ to $B$.
Let
$$\mathbb Z^{X(1)}_{\mathrm{pos}} := \left\{ F \in \mathbb Z^{X(1)} \ \middle| \ \sum_{\rho \in X(1)} F(\rho) \geq 0 \right\},$$
where $X(1)$ is the set of rays in $X$.
Then $\ZXpos$ is a semiring, where the addition and the multiplication are the pointwise max-plus operation.
For any ray $\rho \in X(1)$, let $w_{\rho}$ be its weight and $\bfit d_{\rho}$ be its primitive direction vector.
The homomorphism $\phi_X$ is defined as
$$\phi_X : \Bxpmfcn \to \ZXpos, \quad f \mapsto \phi_X(f),$$
where $\phi_X(f)$ is defined as $\phi_X(f)(\rho) = w_\rho f(\bfit d_{\rho})$, where $w_{\rho}$ is the weight of $\rho$ and $\bfit d_{\rho}$ is the primitive direction vector of $\rho$.
We also see that $\Ker (\phi_X) = \CX$, and then the semiring $\Bxpmfcn / \CX$ is isomorphic to a subsemiring of $\ZXpos$.

We make the correspondence $X \mapsto \phi_X$ into a contravariant functor.
To do this, we need to define the morphisms of semiring homomorphisms.
Let $A, B$ be finite sets.
Let $\phi : \Bxnpmfcn \to \ZApos$ and $\psi : \Bympmfcn \to \ZBpos$ be semiring homomorphisms.
Then a morphism from $\phi$ to $\psi$ is any semiring homomorphism from $\Im(\phi)$ to $\Im(\psi)$.
The reason why we use such definition is described in Section \ref{subsec:MorDef}.
Let $X,Y$ be $1$-dimensional tropical fans.
For a morphism $\mu : X \to Y$, we construct a morphism $\mu^* : \phi_Y \to \phi_X$.
Thus the correspondence $X \mapsto \phi_X, \mu \mapsto \mu^*$ define a contravariant functor.
It is shown that the functor is faithful.
We also show that a morphism $\nu : \phi_Y \to \phi_X$ is of the form $\mu^*$ for some $\mu : X \to Y$ if and only if $\nu$ has a certain property, called \textit{geometric}.
Thus we have our second main theorem.

\begin{thm}
The correspondence $X \mapsto \phi_X$ and $\mu \mapsto \mu^*$ gives a fully faithful contravariant functor from the category of 1-dimensional tropical fans to the category of realizable homomorphisms whose morphisms are geometric morphisms,
\end{thm}

where the definition of realizable homomorphism is in Section \ref{subsec:RealHom}.

As an application, we discuss the smoothness of $1$-dimensional tropical fans at the origin.
The following is our third main theorem.

\begin{thm}
Let $X$ be a 1-dimensional tropical fan in $\mathbb R^n$.
Then, $X$ is smooth at $\mathbf 0$ if and only if the map $\phi_X$ is surjective.
\end{thm}

This paper is constructed as follows.
In Section 2, we recall some basic definitions and facts about semirings, tropical algebra and 1-dimensional tropical fans.
We also introduce the $\mathbb T$-extensions of semirings.
In Section 3, we develop the local theory of functions on $\mathbb R^n$ defined by tropical Laurent polynomials, and show our first main theorem.
In Section 4, we apply $\CXp$ to the first main theorem.
In Section 5, we define the weighted evaluation map $\phi_X$ of a 1-dimensional tropical fan $X$, and study its properties.
In Section 6, we make the correspondence $X \mapsto \phi_X$ into a contravariant functor, and study its properties.
In Section 7, we discuss the smoothness of 1-dimensional tropical fans at the origin.

The author is grateful to J. Song for proofreading this manuscript.

\section{Preliminaries}

\subsection{Semirings}

A \textit{semiring} $(R, +, \cdot, 0, 1)$ is a set $R$ equipped with two binary operations $+$ (addition) and $\cdot$ (multiplication) and two fixed elements $0,1$ satisfying the following conditions:
\begin{enumerate}[(1)]
\item $(R,+,0)$ is a commutative monoid,
\item $(R,\cdot,1)$ is a monoid,
\item multiplication is distributive over addition, and
\item $0 \cdot a = 0$ for any $a \in R$.
\end{enumerate}
We also denote this semiring by $(R, +, \cdot)$ or $R$ if it causes no confusions.
Note that the condition $0 \cdot a = 0$ is necessary because it does not follow from the other axioms.
If a semiring $R$ satisfies $a \cdot b = b \cdot a$ for any $a,b \in R$, we say $R$ is \textit{commutative}.
An element $a$ in a semiring $R$ is \textit{invertible} or a \textit{unit} if there exists an element $b \in R$ such that $a \cdot b = b \cdot a =1$.
The set of all invertible elements of a semiring $R$ forms a group with respect to multiplication, which we call the \textit{unit group} of $R$.
A commutative semiring $R$ is called \textit{semifield} if its unit group is $R \smallsetminus \{ 0 \}$.
A subset $S$ of a semiring $R$ is a \textit{subsemiring} of $R$ if $S$ is a semiring with respect to the operations on $R$.
A \textit{subsemifield} is similarly defined.
A semiring $R$ is \textit{idempotent} if $a+a=a$ for any $a \in R$.

Let $A$ be a subset of $R$.
The subsemiring of $R$ \textit{generated by} $A$ is the minimum subsemiring of $R$ including $A$.
It consists of the elements of $R$ of the form
$$a_{11} a_{12} \cdots a_{1k_1} + a_{21} a_{22} \cdots a_{2k_2} + \cdots + a_{l1} a_{l2} \cdots a_{lk_l},$$
for some $a_{11}, \ldots, a_{lk_l} \in A$.
If a subsemiring $S$ of $R$ is generated by $A$, the set $A$ is called a \textit{generating set} of $S$.

A \textit{congruence} $\mathcal C$ on a semiring $R$ is a subset of $R \times R$ such that, for any $a,b,c,d \in R$,
\begin{enumerate}[(C1)]
\item $(a, a) \in {\mathcal C}$,
\item if $(a,b) \in {\mathcal C}$, then $(b,a) \in {\mathcal C}$,
\item if $(a,b), (b,c) \in {\mathcal C}$, then $(a,c) \in {\mathcal C}$, and
\item if $(a,b), (c,d) \in {\mathcal C}$, then $(a+c,b+d), (a \cdot c, b \cdot d) \in {\mathcal C}$.
\end{enumerate}

We often write $a \underset{\mathcal C}{\sim} b$ instead of $(a,b) \in \mathcal C$.
The conditions (C1) -- (C3) mean that a congruence is an equivalence relation.

Let ${\mathcal C}$ be a congruence on a semiring $R$.
For any element $a \in R$, we denote $[a]$ the equivalence class of $a$ in $R/{\mathcal C}$.
For equivalence classes $[a],[b] \in R/{\mathcal C}$, we define
$$[a]+[b] = [a+b], \qquad [a] \cdot [b] = [a \cdot b].$$
These operations are well-defined, and $(R/{\mathcal C}, +, \cdot, [0], [1])$ forms a semiring, which is called \textit{quotient semiring} of $R$ by $\mathcal C$.

A \textit{homomorphism} $\varphi : R_1 \to R_2$ of semirings is a map such that, for any $a,b \in R_1$,
\begin{enumerate}[(1)]
\item $\varphi (a+b) = \varphi(a) + \varphi(b)$,
\item $\varphi(a \cdot b) = \varphi(a) \cdot \varphi(b)$,
\item $\varphi(0) = 0$, and
\item $\varphi(1) = 1$.
\end{enumerate}
A homomorphism $\varphi : R_1 \to R_2$ is an \textit{isomorphism} if and only if there exists a homomorphism $\psi : R_2 \to R_1$ such that $\psi \circ \varphi = \id_{R_1}$ and $\varphi \circ \psi = \id_{R_2}$.

Let $\varphi : R_1 \to R_2$ be a homomorphism of semirings. 
The \textit{kernel} $\textrm{Ker}(\varphi)$ of $\varphi$ is the congruence $\{ (a,b) \in R_1 \times R_1 \ | \ \varphi(a)=\varphi(b) \}$ on $R_1$.
The \textit{image} $\textrm{Im}(\varphi)$ of $\varphi$ is the subsemiring $\{ \varphi(a) \in R_2 \ | \ a \in R_1 \}$ of $R_2$.
The homomorphism $\varphi$ naturally induces an isomorphism $R_1/\Ker (\varphi) \to \Im (\varphi)$.

Let $\varphi : R_1 \to R_2$ be a homomorphism of semirings and $\mathcal C$ be a congruence on $R_2$.
Then the inverse image $\varphi^{-1}(\mathcal C) := \{ (a, b) \in R_1\times R_1 \ | \ (\varphi(a), \varphi(b)) \in \mathcal C \}$ is a congruence on $R_1$.
Moreover, $\varphi^{-1}(\mathcal C)$ is the kernel of the composition of the homomorphisms $R_1 \to R_2 \to R_2/\mathcal C$.
If $\varphi$ is surjective, the isomorphism $R_1 / \varphi^{-1}(\mathcal C) \to R_2 / \mathcal C$ is induced.

Let $\mathcal C_1, \mathcal C_2$ be congruences on a semiring $R$ such that $\mathcal C_1 \subset \mathcal C_2$.
We define the congruence $\mathcal C_2 / \mathcal C_1$ on $R/ \mathcal C_1$ as
$$\mathcal C_2 / \mathcal C_1 = \{ ([a], [b]) \in (R/ \mathcal C_1)^2 \ | \ (a,b) \in \mathcal C_2 \}.$$
There is a canonical isomorphism $R/\mathcal C_2 \to (R/\mathcal C_1) / (\mathcal C_2 / \mathcal C_1)$.

\subsection{Tropical algebra}

For $a, b \in \mathbb T := \mathbb R \cup \{ -\infty \}$, we define
$$a \oplus b = \max \{ a, b \}, \qquad a \odot b = a+b,$$
where, $a + (-\infty) = (-\infty) + a = -\infty$ for any $a \in \mathbb T$.
Then $\mathbb T = (\mathbb T, \oplus, \odot, -\infty, 0)$ forms an idempotent semifield, which we call the \textit{tropical semifield}.
A \textit{tropical Laurent monomial} of $n$ variables $x_1, \ldots, x_n$ is a formal product of the form $a \odot x_1^{u_1} \odot \cdots \odot x_n^{u_n} = a \odot {\bfit x}^{\bfit u}$, where $a \in \mathbb T$ and $\bfit u = (u_1, \ldots, u_n) \in \mathbb Z^n$.
A \textit{tropical Laurent polynomial} of $x_1, \ldots, x_n$ is a formal sum of the form
$$\bigoplus_{\bfit u \in \mathbb Z^n} a_{\bfit u} \odot \bfit x^{\bfit u},$$
where $a_{\bfit u} \in \mathbb T$ and $a_{\bfit u} = 0$ except for finitely many $\bfit u$.
We denote by $\mathbb T[\bfit x^{\pm}] = \mathbb T[x_1^{\pm}, \ldots, x_n^{\pm}]$ the set of all tropical Laurent polynomials of $x_1, \ldots, x_n$.
The addition and multiplication of tropical Laurent polynomials are naturally defined, and $\mathbb T[\bfit x^{\pm}]$ forms a commutative idempotent semiring.

\begin{ex}
For a tropical Laurent polynomial $P \in \Txpm$ and a vector $\bfit a = ( a_1, \ldots, a_n)$ in $\mathbb R^n$, $P(a_1 \odot x_1, \ldots, a_n \odot x_n)$ is also a tropical Laurent polynomial, which we denote by $P(\bfit a + \bfit x)$.
\end{ex}

For a tropical Laurent polynomial $P = \bigoplus_{i=1}^k a_i \odot \bfit x^{\bfit u_i}$ and a vector $\bfit p$, the \textit{value} of $P$ at $\bfit p$ is
$$P(\bfit p) := \bigoplus_{i=1}^k a_i \odot \bfit p^{\bfit u_i} = \max \{ a_i + \bfit u_i \cdot \bfit p \}_{i=1}^k,$$
where $\cdot$ is the standard inner product.

\begin{dfn}
The relation
$$\left\{ (P, Q) \in \left( \mathbb T [ \bfit x^{\pm} ] \right)^2 \ \middle| \ P(\bfit p) = Q(\bfit p) \text{ for any } \bfit p \in \mathbb R^n \right\}$$
on $\mathbb T [ \bfit x^{\pm} ]$ is a congruence.
The \textit{tropical Laurent polynomial function semiring} $\Txpmfcn$ is the quotient semiring of $\mathbb T [ \bfit x^{\pm} ]$ by the above congruence.
The equivalence class of $P$ is denoted by $[P]$.
\end{dfn}

The elements of $\Txpmfcn$ can be regarded as functions on $\mathbb R^n$.
We call the class $[P]$ the function defined by $P$.

The subsemifield $\mathbb B = \{ -\infty, 0 \}$ of $\mathbb T$ is called the \textit{Boolean semifield}.
We define the semirings $\mathbb B [\bfit x^{\pm}]$ and $\Bxpmfcn$ in the same way to the case of $\mathbb T$.

\begin{ex}
Let $P = \bfit x^{\bfit u} = x_1^{u_1} \odot \cdots \odot x_n^{u_n}$ be a Boolean Laurent monomial.
For any vector $\bfit p = (p_1, \ldots, p_n) \in \mathbb R^n$, $P(\bfit p) = \bfit u \cdot \bfit p = u_1p_1 + \cdots u_np_n$.
Thus the function $[P]$ is a linear form on $\mathbb R^n$ with integer coefficient.
Conversely, any linear form on $\mathbb R^n$ with integer coefficient is the function defined by a Boolean Laurent monomial of $n$ variables.
\end{ex}

\begin{lem}
\label{Bfcn}
Let $f \in \Bxpmfcn$ be a function, $t \geq 0$ be a nonnegative real number, and $\bfit p \in \mathbb R^n$ be a vector.
Then $f(t \bfit p) = t f(\bfit p)$, where the multiplication is the standard one.
\end{lem}

\begin{proof}
Let $f = \left[ \bigoplus_{i=1}^n \bfit x^{\bfit u_i} \right]$.
Then
$f(t \bfit p) = \max_i \{ \bfit u_i \cdot (t \bfit p) \} = t \max_i \{ \bfit u_i \cdot \bfit p \} = t f(\bfit p)$.
\end{proof}

\begin{lem}
Let $R$ be an idempotent semiring.
Fix invertible elements $a_1, \ldots, a_n \in R$.
Then there exists a unique homomorphism $\phi : \Bxpm = \mathbb B[x_1^{\pm}, \ldots, x_n^{\pm}] \to R$ such that $\phi(x_i) = a_i$ for any $i$. 
\end{lem}

\begin{proof}
Since $\Bxpm$ is generated by $\{ x_1, x_1^{-1}, \ldots, x_n, x_n^{-1} \}$, such $\phi$ is unique if it exists, that is, $\phi$ must be defined as $\phi(P) = P(a_1, \ldots, a_n)$.
It is easily checked that this $\phi$ is a semiring homomorphism.
\end{proof}

\begin{lem}
\label{universal1}
Fix monomials $\bfit y^{\bfit u_1}, \ldots, \bfit y^{\bfit u_n} \in \Bypm = \mathbb B[y_1^{\pm}, \ldots, y_m^{\pm}]$.
Then there exists a unique semiring homomorphism $\phi : \Bxpmfcn = \mathbb B[x_1^{\pm}, \ldots, x_n^{\pm}]_{\mathrm {fcn}} \to \Bypmfcn$ such that $\phi([x_i]) = [\bfit y^{\bfit u_i}]$ for any $i$.
\end{lem}

\begin{proof}
Since $\Bxpmfcn$ is generated by $\{ [x_1], [x_1^{-1}], \ldots, [x_n], [x_n^{-1}] \}$, such $\phi$ is unique if it exists.
We show the existence.
Since $x_1, \ldots, x_n$ are invertible in $\Bxpm$, there exists a unique semiring homomorphism $\phi_0 :  \Bxpm \to \Bypm$ such that $\phi_0(x_i) = \bfit y^{\bfit u_i}$ for $i=1, \ldots, n$.
We show that $\phi_0$ induces a homomorphism $\Bxpmfcn \to \Bypmfcn$, i.e., $[\phi_0(P)] = [\phi_0(Q)]$ if $[P] = [Q]$.
Let $P, Q \in \Bxpm$ be polynomials such that $[P] = [Q]$.
Since
$$\phi_0(P(x_1, \ldots, x_n)) = P(\bfit y^{\bfit u_1}, \ldots, \bfit y^{\bfit u_n}) \text{ and }
\phi_0(Q(x_1, \ldots, x_n)) = Q(\bfit y^{\bfit u_1}, \ldots, \bfit y^{\bfit u_n}),$$
for any $\bfit p \in \mathbb R^m$,
\[ \phi_0(P)(\bfit p) = P(\bfit u_1 \cdot \bfit p, \ldots, \bfit u_n \cdot \bfit p) = Q(\bfit u_1 \cdot \bfit p, \ldots, \bfit u_n \cdot \bfit p) = \phi_0(Q)(\bfit p). \qedhere \]
\end{proof}

\subsection{One dimensional tropical fans}

A \textit{ray} in $\mathbb R^n$ is the subset of $\mathbb R^n$ of the form $\{ t \bfit d \ | \ t \geq 0 \}$ for some vector $\bfit d \in \mathbb R^n \smallsetminus \{ \mathbf 0 \}$. The vector $\bfit d$ is called a \textit{direction vector} of the ray.
For a ray $\rho$ and a direction vector $\bfit d$ of $\rho$, we say $\bfit d$ \textit{spans} $\rho$.
A ray is \textit{rational} if it is spanned by an integer vector.
Every rational ray has a unique primitive direction vector.
A \textit{1-dimensional fan} in $\mathbb R^n$ is a set of the form $\{ \{ {\mathbf 0} \}, \rho_1, \rho_2, \ldots, \rho_r \}$ for some positive integer $r$ and rational rays $\rho_1, \ldots, \rho_r$ in $\mathbb R^n$.
For a 1-dimensional fan $X$, we denote $X(1)$ the set of all the rays in $X$.
The \textit{support} $|X|$ of a 1-dimensional fan $X$ is the union of all the rays in $X$.
For two 1-dimensional fans $X \subset \mathbb R^n$ and $Y \subset \mathbb R^m$, a \textit{morphism} $\mu:X \to Y$ is a map from $|X|$ to $|Y|$ which is the restriction of a linear map from $\mathbb R^n$ to $\mathbb R^m$ defined by an integer matrix.
A \textit{weighted 1-dimensional fan} $(X, \omega_X)$ is a pair of a 1-dimensional fan $X$ and a map $\omega_X : X(1) \to \mathbb Z_{> 0}$.
We often write $X = (X, \omega_X)$.
We call $\omega_X(\rho)$ the \textit{weight} of $\rho$.

\begin{dfn}
Let $X = (X, \omega_X)$ be a weighted 1-dimensional fan.
Let $\rho_1, \ldots, \rho_r$ be all the rays in $X$ and $\bfit d_i$ be the primitive direction vector of $\rho_i$ for $i=1, \ldots, r$.
We denote $w_i := \omega_X(\rho_i)$.
Then $X$ is a \textit{1-dimensional tropical fan} if it satisfies the following condition, called \textit{balancing condition}:
$$w_1 \bfit d_1 + \cdots + w_r \bfit d_r = \mathbf 0.$$
A morphism of 1-dimensional tropical fans is that of 1-dimensional fans.
\end{dfn}

\begin{rem}
See \cite{GKM} for the definition of tropical fans of general dimensions.
Briefly, a tropical variety is an object obtained by gluing tropical fans.
See \cite{AR} for precise definition.
\end{rem}

\subsection{$\mathbb T$-extension}

In this subsection, we introduce the $\mathbb T$-extension of a semiring.
Some important semirings in this paper have the structure of $\mathbb T$-extension.

Fix a semiring $R = (R, +, \cdot, 0, 1)$.
Assume that $R \neq \{ 0 \}$, and that $f+g \neq 0$ and $f \cdot g \neq 0$ for any $f, g \in R \smallsetminus \{ 0 \}$.

\begin{dfn}
\label{Textdef}
The $\mathbb T$\textit{-extension} $R \times_e \mathbb T$ of $R$ is the semiring defined as follows:
As a set, $R \times_e \mathbb T = ((R \smallsetminus \{ 0 \}) \times \mathbb R) \cup \{ -\infty \}$.
For $f,g \in R \smallsetminus \{ 0 \}$ and $a,b \in \mathbb R$, we define
$$(f, a) + (g, b) =
\begin{cases}
(f, a) &\text{ if } a>b \\
(g, b) &\text{ if } a<b \\
(f+g, a) &\text{ if } a=b,
\end{cases}$$
$$(f, a) \cdot (g, b) = (f \cdot g, a+b).$$
Also we define $(f, a) + (-\infty) = (-\infty) + (f, a) = (f, a)$ for any $f \in R \smallsetminus \{ 0 \}$ and $a \in \mathbb R$, and $(-\infty) + (-\infty) = -\infty$.
The product of $-\infty$ and any element is $-\infty$.
\end{dfn}

It is easily checked that $R \times_e \mathbb T$ forms a semiring.
We give some properties of $\mathbb T$-extensions.

\begin{lem}
\label{Textcong}
A congruence $\mathcal C$ on $R \times_e \mathbb T$ satisfies just one of the following properties: 
\begin{enumerate}[$(1)$]
\item If $(f, a) \simwrt{\mathcal C} (g, b)$, then $a=b$.
\item $(f, a) \simwrt{\mathcal C} (g, b)$ for any $f, g \in R \smallsetminus \{ 0 \}$ and $a,b \in \mathbb R$.
\end{enumerate}
\end{lem}

\begin{proof}
In this proof, we omit $\mathcal C$ in $\simwrt{\mathcal C}$.
There is no congruence satisfying (1) and (2) because, if there exists such a congruence, for any $f \in R \smallsetminus \{ 0 \}$, $(f, 0) \not\sim (f, 1)$ by (1) and $(f, 0) \sim (f, 1)$ by (2), which is a contradiction.
Thus it is sufficient to show that if $\mathcal C$ does not satisfy $(1)$, then it satisfies $(2)$.
Assume that there exist $f, g \in R \smallsetminus \{ 0 \}$ and $a,b \in \mathbb R$ such that $(f, a) \sim (g, b)$ and $a < b$.
Let $I \subset \mathbb R$ be an interval.
Consider the following statement:
$$(*) \quad (h_1, c_1) \sim (h_2, c_2) \text{ for any } h_1, h_2 \in R \smallsetminus \{ 0 \} \text{ and } c_1, c_2 \in I.$$
Our purpose is to show that $(*)$ holds for $I = \mathbb R$.

For any $h \in R \smallsetminus \{ 0 \}$ and $c \in \mathbb R$ such that $a<c<b$,
$$(h, c) = (h, c) + (f, a) \sim (h, c) + (g, b) = (g, b).$$
Thus for any $h_1, h_2 \in R \smallsetminus \{ 0 \}$ and $c_1, c_2 \in \mathbb R$ such that $a<c_1, c_2<b$,
$$(h_1, c_1) \sim (g, b) \sim (h_2, c_2).$$
Hence $(*)$ holds for $I=(a,b)$.

Next we show that $(*)$ holds for the interval $(a+ \epsilon, b + \epsilon)$ for any $\epsilon \in \mathbb R$.
For any $h_1, h_2 \in R \smallsetminus \{ 0 \}$ and $c_1, c_2 \in \mathbb R$ such that $a+\epsilon<c_1, c_2<b+\epsilon$,
$$(h_1, c_1) = (h_1, c_1-\epsilon) \cdot (1, \epsilon) \sim (h_2, c_2-\epsilon) \cdot (1, \epsilon) = (h_2, c_2),$$
then $(*)$ holds for $(a+ \epsilon, b + \epsilon)$.

Let $I_1$ and $I_2$ be two intervals such that $I_1 \cap I_2 \neq \emptyset$. 
Assume that $(*)$ holds for $I_1$ and $I_2$.
Thus $(*)$ holds for $I_1 \cup I_2$ because, for any $h_1, h_2 \in R \smallsetminus \{ 0 \}, c_1, c_2 \in I_1 \cup I_2$, and $c_3 \in I_1 \cap I_2$,
$$(h_1, c_1) \sim (1, c_3) \sim (h_2, c_2).$$

Combining the results so far, it is shown that $(*)$ holds for any bounded open interval.
Finally, for any $h_1, h_2 \in R \smallsetminus \{ 0 \}$ and $c_1, c_2 \in \mathbb R$, $(h_1, c_1) \sim (h_2, c_2)$ since $(*)$ holds for $(c_1-1, c_2+1)$.
It means that $(*)$ holds for $I= \mathbb R$.
\end{proof}

We say a congruence on $R \times_e \mathbb T$ is \textit{partial} if it satisfies the condition (1) in the previous lemma. 
Note that any partial congruence $\mathcal C$ on $R \times_e \mathbb T$ has the following property:
$$(f, a) \nsimwrt{\mathcal C} -\infty \text{ for any } f \in R \smallsetminus \{ 0 \} \text{ and } a \in \mathbb R.$$
Indeed, if $(f, a) \simwrt{\mathcal C} -\infty$,
$$(f, a) = (f,a) + (1,a-1) \simwrt{\mathcal C} (-\infty) + (1, a-1) = (1, a-1).$$
It contradicts to the assumption that $\mathcal C$ is partial.

\begin{lem}
\label{Textcongr}
There is a bijection between the set of all partial congruences on $R \times_e \mathbb T$ and the set of all congruences $\mathcal C$ on $R$ which satisfy the following condition:
$$(*) \quad f \nsimwrt{\mathcal C} 0 \ \text{ for any } \ f \in R \smallsetminus \{ 0 \}.$$
For a partial congruence $\mathcal C$ on $R \times_e \mathbb T$, the corresponding congruence on $R$ is
$$\Phi(\mathcal C) = \{ (f, g) \in (R \smallsetminus \{ 0 \})^2 \ | \ ((f, 0), (g, 0)) \in \mathcal C \} \cup \{ (0,0) \}.$$
For a congruence $\mathcal C'$ on $R$ which satisfies $(*)$, the corresponding congruence on $R \times_e \mathbb T$ is
$$\Psi(\mathcal C') = \{ ((f, a), (g, a)) \in ((R \smallsetminus \{ 0 \}) \times \mathbb R)^2 \ | \ (f, g) \in \mathcal C', a \in \mathbb R \} \cup \{ (-\infty, -\infty) \}.$$
\end{lem}

\begin{proof}
Step1. $\Phi(\mathcal C)$ is a congruence on $R$ which satisfies $(*)$.

It is easy to check that $\Phi(\mathcal C)$ is an equivalence relation.
Let us show that (C4) holds.
If $f_1 \simwrt{\Phi(\mathcal C)} g_1, f_2 \simwrt{\Phi(\mathcal C)} g_2$ for $f_1, f_2, g_1, g_2 \in R \smallsetminus \{ 0 \}$, then $(f_1, 0) \simwrt{\mathcal C} (g_1,0), (f_2, 0) \simwrt{\mathcal C} (g_2,0)$ by the definition of $\Phi(\mathcal C)$.
Thus
$$(f_1 + f_2, 0) = (f_1, 0) + (f_2, 0) \simwrt{\mathcal C} (g_1, 0) + (g_2, 0) = (g_1+g_2, 0),$$
$$(f_1 \cdot f_2, 0) = (f_1, 0) \cdot (f_2, 0)  \simwrt{\mathcal C} (g_1, 0) \cdot (g_2, 0) = (g_1 \cdot g_2, 0),$$
and hence $f_1 + f_2 \simwrt{\Phi(\mathcal C)} g_1 + g_2$ and $f_1 \cdot f_2 \simwrt{\Phi(\mathcal C)} g_1 \cdot g_2$.
The remaining things to show are the following:
\begin{itemize}
\item If $f \simwrt{\Phi(\mathcal C)} g$, then $f + 0 \simwrt{\Phi(\mathcal C)} g + 0 \text{ and } f \cdot 0 \simwrt{\Phi(\mathcal C)} g \cdot 0$,
\item $0 + 0 \simwrt{\mathcal C} 0 + 0 \text{ and } 0 \cdot 0 \simwrt{\mathcal C} 0 \cdot 0.$
\end{itemize}
These are clear.
Hence (C4) holds.
Finally, $\Phi(\mathcal C)$ satisfies $(*)$ by the definition of $\Phi(\mathcal C)$.\\

\noindent
Step2. $\Psi(\mathcal C')$ is a partial congruence on $R \times_e \mathbb T$.

It is easy to check that $\Psi(\mathcal C')$ is an equivalence relation.
Let us show that (C4) holds.
If $(f_1, a_1) \simwrt{\Psi(\mathcal C')} (g_1, a_1), (f_2, a_2) \simwrt{\Psi(\mathcal C')} (g_2,a_2)$ for $f_1, f_2, g_1, g_2 \in R \smallsetminus \{ 0 \}$ and $a_1, a_2 \in \mathbb R$, then $f_1 \simwrt{\mathcal C'} g_1, f_2 \simwrt{\mathcal C'} g_2$ by the definition of $\Psi(\mathcal C')$.
Thus $f_1 \cdot f_2 \simwrt{\mathcal C'} g_1 \cdot g_2$, and hence
$$(f_1, a_1) \cdot (f_2, a_2) = (f_1 \cdot f_2, a_1 +a_2) \simwrt{\Psi(\mathcal C')} (g_1 \cdot g_2, a_1+ a_2) = (g_1, a_1) \cdot (g_2, a_2).$$
As for the addition, if $a_1 > a_2$, 
$$(f_1, a_1) + (f_2, a_2) = (f_1, a_1) \simwrt{\Psi(\mathcal C)} (g_1, a_1) = (g_1, a_1) + (g_2, a_2).$$
The case $a_1 < a_2$ is similar.
If $a_1 = a_2$, since $f_1 + f_2  \simwrt{\mathcal C'} g_1+g_2$,  
$$(f_1, a_1) + (f_2, a_2) = (f_1 + f_2, a_1)  \simwrt{\Psi(\mathcal C')} (g_1 + g_2, a_1) = (g_1, a_1) + (g_2, a_2).$$
The remaining things to show are the following:
\begin{itemize}
\item If $f \simwrt{\Psi(\mathcal C')} g$, then $(f, a) + (-\infty) \simwrt{\Psi(\mathcal C')} (g,a) + (-\infty) \text{ and } (f, a) \cdot (-\infty) \simwrt{\Psi(\mathcal C')} (g, a) \cdot (-\infty)$,
\item $(-\infty) + (-\infty) \simwrt{\Psi(\mathcal C')} (-\infty) + (-\infty) \text{ and } (-\infty) \cdot (-\infty) \simwrt{\Psi(\mathcal C')} (-\infty) \cdot (-\infty).$
\end{itemize}
These are clear.
Hence (C4) holds.
Finally, $\Psi(\mathcal C')$ is obviously partial.\\

\noindent
Step3. $\Psi(\Phi(\mathcal C)) = \mathcal C$ for any partial congruence $\mathcal C$ on $R \times_e \mathbb T$.

If $(f, a) \simwrt{\Psi(\Phi(\mathcal C))} (g, a)$ for $f, g \in R \smallsetminus \{ 0 \}$ and $a \in \mathbb R$, then $f \simwrt{\Phi(\mathcal C)} g$, hence $(f, 0) \simwrt{\mathcal C} (g, 0)$, therefore $(f, a) = (f, 0) \cdot (1, a) \simwrt{\mathcal C} (g, 0) \cdot (1, a) = (g, a)$.
It means that $\Psi(\Phi(\mathcal C)) \subset \mathcal C$.
Conversely, if $(f, a) \simwrt{\mathcal C} (g, a)$ for $f, g \in R \smallsetminus \{ 0 \}$ and $a \in \mathbb R$, then $(f, 0) = (f, a) \cdot (1, -a) \simwrt{\mathcal C} (g, a) \cdot (1, -a) = (g, 0)$, hence $f \simwrt{\Phi(\mathcal C)}  g$, therefore $(f, a) \simwrt{\Psi(\Phi(\mathcal C))} (g, a)$.
It means that $\Psi(\Phi(\mathcal C)) \supset \mathcal C$.\\

\noindent
Step4. $\Phi(\Psi(\mathcal C')) = \mathcal C'$ for any congruence $\mathcal C'$ on $R$ which satisfies $(*)$.

If $f \simwrt{\Phi(\Psi(\mathcal C'))} g$ for $f, g \in R \smallsetminus \{ 0 \}$, then $(f, 0) \simwrt{\Psi(\mathcal C')} (g, 0)$, hence $f \simwrt{\mathcal C'} g$.
It means that $\Phi(\Psi(\mathcal C')) \subset \mathcal C'$.
Conversely, if $f \simwrt{\mathcal C'} g$ for $f, g \in R \smallsetminus \{ 0 \}$, then $(f, 0) \simwrt{\Psi(\mathcal C')} (g, 0)$, hence $f \simwrt{\Phi(\Psi(\mathcal C'))} g$.
It means that $\Phi(\Psi(\mathcal C')) \supset \mathcal C'$.
\end{proof}

\begin{lem}
\label{Textquot}
Let $\mathcal C$ be a partial congruence on $R \times_e \mathbb T$ and $\mathcal C'$ be the congruence on $R$ corresponding to $\mathcal C$ via the map in the previous lemma.
Then $(R \times_e \mathbb T) / \mathcal C$ is isomorphic to $(R / \mathcal C') \times_e \mathbb T$.
\end{lem}

\begin{proof}
First we check that $R/ \mathcal C'$ satisfies the assumption to define its $\mathbb T$-extension.
Since $C'$ satisfies the condition $(*)$ in the previous lemma, we have $1 \nsimwrt{\mathcal C'} 0$.
Thus $R/ \mathcal C' \neq \{ 0 \}$.
For any $[f], [g] \in (R/\mathcal C') \smallsetminus \{ [0] \}$, we have $f \neq 0$ and $g\neq 0$, and then $f+g \neq 0$ and $f \cdot g \neq 0$, hence $[f]+[g] = [f+g] \neq [0]$ and $[f] \cdot [g] = [f \cdot g] \neq [0]$ by $(*)$.

Consider the map
$$\pi:R \times_e \mathbb T \to (R/\mathcal C') \times_e \mathbb T, (f,a) \mapsto ([f], a), -\infty \mapsto -\infty.$$
It is easily checked that $\pi$ is surjective homomorphism.
Thus it is enough to show that $\Ker(\pi) = \mathcal C$.
Clearly $\pi((f,a)) \neq \pi(-\infty)$ for any $f \in R \smallsetminus \{ 0 \}$ and $a \in \mathbb R$.
Thus the elements of $\Ker(\pi)$ are $(-\infty, -\infty)$ or of the form $((f, a), (g, b))$ for some $f, g \in R \smallsetminus \{ 0 \}$ and $a, b \in \mathbb R$.
If $\pi((f, a)) = \pi((g,b))$ for $f, g \in R \smallsetminus \{ 0 \}$ and $a, b \in \mathbb R$,
then $[f]=[g]$ and $a=b$, hence $f \simwrt{\mathcal C'} g$.
It means that $(f, 0) \simwrt{\mathcal C} (g, 0)$ and then $(f, a) = (f, 0) \cdot (1,a) \simwrt{\mathcal C} (g,0) \cdot (1, b) = (g, b)$.
This shows $\Ker(\pi) \subset \mathcal C$.
Conversely, if $(f, a) \simwrt{\mathcal C} (g, b)$ for $f, g \in R \smallsetminus \{ 0 \}$ and $a, b \in \mathbb R$, then $a=b$ since $\mathcal C$ is partial.
As $(f, 0) = (f, a) \cdot (1, -a) \simwrt{\mathcal C} (g, b) \cdot (1, -b) = (g, 0)$, we have $f \simwrt{\mathcal C'} g$.
Therefore $\pi((f, a)) = ([f], a) = ([g], b) = \pi((g, b))$.
This shows $\Ker(\pi) \supset \mathcal C$.
\end{proof}

\section{Local theory of tropical Laurent polynomials}

In the rest of paper, we fix a positive integer $n$.
The notations $\Txpm$ and $\Bxpm$ are always mean $\mathbb T[x_1^{\pm}, \ldots, x_n^{\pm}]$ and $\mathbb B[x_1^{\pm}, \ldots, x_n^{\pm}]$ respectively.

In this section, we fix a vector $\bfit p = (p_1, \ldots, p_n) \in \mathbb R^n$.
Let $\Cp$ be the congruence
$$\left\{ (f, g)  \ \middle| \ 
\begin{aligned}
&\text{ there exists an open neighborhood $U$ of $\bfit p$ } \\
&\text{ such that }  f|_U = g|_U
\end{aligned} \right\}$$
on $\Txpmfcn$, where the topology on $\mathbb R^n$ is the Euclidean topology.
The purpose of this section is to show that the semiring $\Txpmfcn / \Cp$ is isomorphic to the $\mathbb T$-extension of $\Bxpmfcn$.

\subsection{Initial forms of tropical Laurent polynomials}

\begin{dfn}
For a tropical Laurent polynomial $P = \bigoplus_{\bfit u} a_{\bfit u} \odot \bfit x^{\bfit u} \in \mathbb T[\bfit x^{\pm}]$, the \textit{initial form} $\inp (P)$ of $P$ at $\bfit p$ is defined as
$$\inp(P) = \bigoplus_{a_{\bfit u} + \bfit u \cdot \bfit p = P(\bfit p)} a_{\bfit u} \odot \bfit x^{\bfit u}.$$
In other words, $\inp(P)$ is the sum of the terms of $P$ which attain the maximum of $\bigoplus_{\bfit u} a_{\bfit u} \odot \bfit p^{\bfit u} = \max_{\bfit u} \{ a_{\bfit u} + \bfit u \cdot \bfit p \}$.
\end{dfn}

\begin{ex}
\label{exin0}
In the case $\bfit p = {\mathbf 0} = (0, \ldots, 0)$, for a tropical Laurent polynomial $P \in \Txpm$, ${\mathrm {in}}_{\mathbf 0}(P)$ is the sum of the terms of $P$ which have maximum coefficient.
For example, in $\mathbb T[x^{\pm}, y^{\pm}]$,
$${\mathrm {in}}_{\mathbf 0}(1 \oplus (3 \odot x) \oplus (2 \odot y) \oplus (3 \odot x \odot y)) = (3 \odot x) \oplus (3 \odot x \odot y).$$
\end{ex}

Although the map $\inp : \Txpm \to \Txpm$ is not a homomorphism of semirings (see the following remark), we use the following notations:
$$\Ker(\inp) := \left\{ (P,Q) \ \middle| \ \inp(P)=\inp(Q) \right\}, \qquad \Im(\inp) := \left\{ \inp(P) \ \middle| \ P \in \Txpm \right\}.$$
Clearly $\Ker(\inp)$ is an equivalence relation on $\Txpm$.
Thus there is a bijection from $\Txpm / \Ker(\inp)$ to $\Im(\inp)$ induced by $\inp$.
Later we see that $\Ker(\inp)$ is in fact a congruence.

\begin{rem}
The map $\inp:\Txpm \to \Txpm$ is not a homomorphism of semirings.
For instance, in $\mathbb T[x^{\pm}]$, ${\mathrm {in}}_0(1 \oplus x) = 1$, while ${\mathrm {in}}_0(1) \oplus {\mathrm {in}}_0(x) = 1 \oplus x$.
Moreover, $\Im(\inp)$ is not a subsemiring of $\Txpm$.
For instance, in $\mathbb T[x^{\pm}]$, $1$ and $x$ are in $\Im ({\mathrm {in}}_0)$, while $1 \oplus x \not\in \Im ({\mathrm {in}}_0)$ because a tropical Laurent polynomial which has at least two terms with different coefficients is not in $\mathrm{Im}(\mathrm{in}_0)$ by Example \ref{exin0}.
\end{rem}

\begin{lem}
\label{lem:inp add mult}
For any $P, Q \in \Txpm$,
$$\inp(P \oplus Q) =
\begin{cases}
\inp(P) &(P(\bfit p) > Q(\bfit p))\\
\inp(Q) &(P(\bfit p) < Q(\bfit p))\\
\inp(P) \oplus \inp(Q) &(P(\bfit p) = Q(\bfit p)),
\end{cases}$$
and
$$\inp(P \odot Q) = \inp(P) \odot \inp(Q).$$
\end{lem}

\begin{proof}
Let $P= \bigoplus_{i=1}^k a_i \odot \bfit x^{\bfit u_i}$ and $Q= \bigoplus_{j=1}^l b_j \odot \bfit x^{\bfit v_j}$.
We may assume that $\inp (P) = \bigoplus_{i=1}^{k'} a_i \odot \bfit x^{\bfit u_i}$ and $\inp (Q) = \bigoplus_{j=1}^{l'} b_j \odot \bfit x^{\bfit v_j}$ for some $k', l'$ without loss of generality.
By the definition of initial forms, the maximum of $\{ a_i + \bfit u_i \cdot \bfit p \}_{i=1}^k$ is attained at $i=1, \ldots, k'$, and the maximum of $\{ b_j + \bfit v_j \cdot \bfit p \}_{j=1}^l$ is attained at $j=1, \ldots, l'$.

We now consider the initial form of
$$P \oplus Q = \bigoplus_{i=1}^k a_i \odot \bfit x^{\bfit u_i} \oplus \bigoplus_{j=1}^l b_j \odot \bfit x^{\bfit v_j}$$
at $\bfit p$.
If $P(\bfit p) > Q(\bfit p)$, the maximum of $\{ a_i + \bfit u_i \cdot \bfit p \}_{i=1}^k \cup \{ b_j + \bfit v_j \cdot \bfit p \}_{j=1}^l$ is attained at $i=1, \ldots, k'$.
Then $\inp(P \oplus Q) = \inp(P)$.
The case $P(\bfit p) < Q(\bfit p)$ is similar.
If $P(\bfit p) = Q(\bfit p)$, the maximum of $\{ a_i + \bfit u_i \cdot \bfit p \}_{i=1}^k \cup \{ b_j + \bfit v_j \cdot \bfit p \}_{j=1}^l$ is attained at $i=1, \ldots, k'$ and $j=1, \ldots, l'$.
Then $\inp(P \oplus Q) = \inp(P) \oplus \inp(Q)$.

We next consider the initial form of
$$P \odot Q = \bigoplus_{i=1}^k \bigoplus_{j=1}^l (a_i+b_j) \odot \bfit x^{\bfit u_i + \bfit v_j}$$
at $\bfit p$.
Since $a_i + b_j + (\bfit u_i + \bfit v_j) \cdot \bfit p = (a_i + \bfit u_i \cdot \bfit p) + (b_j + \bfit v_j \cdot \bfit p)$, the maximum of $\{ a_i + b_j + (\bfit u_i + \bfit v_j) \cdot \bfit p \}_{i,j}$ is attained at $(i, j)$ for any $i=1, \ldots, k'$ and $j=1, \ldots, l'$.
Thus
\[ \inp(P \odot Q) = \bigoplus_{i=1}^{k'} \bigoplus_{j=1}^{l'} (a_i+b_j) \odot \bfit x^{\bfit u_i + \bfit v_j} = \inp(P) \odot \inp(Q). \qedhere \]
\end{proof}

\begin{lem}
For any $P \in \Txpm$, $\inp(P)(\bfit p) = P(\bfit p)$.
Moreover, for any $P, Q \in \Txpm$ such that $\inp(P) = \inp(Q)$, we have $P(\bfit p) = Q(\bfit q)$.
\end{lem}

\begin{proof}
The former immediately follows from the definition of initial forms.
The latter follows from the former.
\end{proof}

\begin{prop}
The relation $\Ker(\inp)$ is a congruence.
\end{prop}

\begin{proof}
Let $P_1, P_2, Q_1, Q_2 \in \Txpm$ be tropical Laurent polynomials such that the pairs $(P_1, P_2), (Q_1, Q_2)$ are in $\Ker(\inp)$.

We show that $(P_1 \oplus Q_1, P_2 \oplus Q_2) \in \Ker(\inp)$.
Note that $P_1(\bfit p) = P_2(\bfit p)$ and $Q_1(\bfit p) = Q_2(\bfit p)$ by the previous lemma.
We use Lemma \ref{lem:inp add mult}.
If $P_1(\bfit p) > Q_1(\bfit p)$,
$$\inp(P_1 \oplus Q_1) = \inp(P_1) = \inp(P_2) = \inp(P_2 \oplus Q_2).$$
The case $P_1(\bfit p) < Q_1(\bfit p)$ is similar.
If $P_1(\bfit p) = Q_1(\bfit p)$,
$$\inp(P_1 \oplus Q_1) = \inp(P_1) \oplus \inp(Q_1) = \inp(P_2) \oplus \inp(Q_2) = \inp(P_2 \oplus Q_2).$$
Hence $(P_1 \oplus Q_1, P_2 \oplus Q_2) \in \Ker(\inp)$.

We also have $(P_1 \odot Q_1, P_2 \odot Q_2) \in \Ker(\inp)$ since
\[ \inp(P_1 \odot Q_1) = \inp(P_1) \odot \inp(Q_1) = \inp(P_2) \odot \inp(Q_2) = \inp(P_2 \odot Q_2). \qedhere \]
\end{proof}

By the above proposition, the quotient set $\Txpm / \Ker(\inp)$ has a semiring structure.
Via the bijection $\Txpm / \Ker(\inp) \to \Im(\inp)$, the semiring structure of $\Im(\inp)$ is defined.
We denote by $\oplus_{\bfit p}$ and $\odot_{\bfit p}$ the addition and multiplication on $\Im(\inp)$ defined as above.
That is, for any $P, Q \in \Txpm$,
$$\inp(P) \oplus_{\bfit p} \inp(Q) = \inp(P \oplus Q), \qquad \inp(P) \odot_{\bfit p} \inp(Q) = \inp(P \odot Q).$$

In the rest of paper, $\Im(\inp)$ always means the semiring $(\Im(\inp), \oplus_{\bfit p}, \odot_{\bfit p})$.

\begin{lem}
\label{lem:inpinp}
For any $P \in \Txpm$,
\begin{enumerate}[$(1)$]
\item $\inp(\inp (P)) = \inp(P)$, and
\item $P \in \Im(\inp)$ if and only if $\inp(P)=P$.
\end{enumerate}
\end{lem}

\begin{proof}
(1) is clear by the definition of initial forms.
We show (2).
If $P \in \Im(\inp)$, $\inp(Q) = P$ for some $Q \in \Txpm$.
Then
$$\inp(P) = \inp(\inp(Q)) = \inp(Q) = P.$$
The converse is clear.
\end{proof}

\begin{prop}
For any $P,Q \in \Im(\inp)$,
$$P \oplus_{\bfit p} Q =
\begin{cases}
P &(P(\bfit p) > Q(\bfit p))\\
Q &(P(\bfit p) < Q(\bfit p))\\
P \oplus Q &(P(\bfit p) = Q(\bfit p)),
\end{cases}$$
and
$$P \odot_{\bfit p} Q = P \odot Q.$$
In particular, $\odot_{\bfit p}$ is same as $\odot$.
\end{prop}

\begin{proof}
This immediately follows from Lemma \ref{lem:inp add mult} and Lemma \ref{lem:inpinp}
\end{proof}

We show that $\Im(\inp)$ is isomorphic to the $\mathbb T$-extension of $\Bxpm$.
For any tropical Laurent polynomial $P = \bigoplus_i a_i \odot \bfit x^{\bfit u_i} \in \Txpm$, we denote $P_{\mathbb B} := \bigoplus_i \bfit x^{\bfit u_i} \in \Bxpm$.
It is easily checked that
$$(P \oplus Q)_{\mathbb B} = P_{\mathbb B} \oplus Q_{\mathbb P}, \qquad (P \odot Q)_{\mathbb B} = P_{\mathbb B} \odot Q_{\mathbb B}.$$
We use the following lemma.

\begin{lem}
\label{inpP1}
For any $P \in \Im(\inp)$,
$$P(\bm x) = P(\bfit p) \odot P_{\mathbb B} (\bfit x - \bfit p).$$
\end{lem}

\begin{proof} 
Let $P(\bfit x)= \bigoplus_{i=1}^k a_i \odot \bfit x^{\bfit u_i}$ and $\bfit u_i = (u_{i1}, \ldots, u_{in})$ for any $i=1, \ldots, k$.
Since $P \in \Im(\inp)$, $P(\bfit p) = a_i + \bfit u_i \cdot \bfit p$ for any $i=1, \ldots, k$.
Thus
\begin{align*}
P(\bfit p) \odot P_{\mathbb B} (\bfit x - \bfit p) &= \bigoplus_{i=1}^k P(\bfit p) \odot ((-p_1) \odot x_1)^{u_{i1}} \odot \cdots \odot  ((-p_n) \odot x_n)^{u_{in}}\\
&= \bigoplus_{i=1}^k P(\bfit p) \odot (-\bfit p)^{\bfit u_i} \odot \bfit x^{\bfit u_i}\\
&= \bigoplus_{i=1}^k (a_i + \bfit u_i \cdot \bfit p + \bfit u_i \cdot (- \bfit p)) \odot \bfit x^{\bfit u_i}\\
&= \bigoplus_{i=1}^k a_i \odot \bfit x^{\bfit u_i} = P(\bfit x). \qedhere
\end{align*}
\end{proof}

\begin{prop}
\label{TtoB}
The map
$$\Phi:\Im(\inp) \to \Bxpm \times_e \mathbb T, \quad P \mapsto \begin{cases}
(P_{\mathbb B}, P(\bfit p)) &\text{ if } P \neq -\infty \\
-\infty &\text{ if } P = -\infty
\end{cases}$$
is an isomorphism.
\end{prop}

\begin{proof}
Step1. $\Phi$ is a homomorphism.

We first show $\Phi(P \oplus_{\bfit p} Q) = \Phi(P) + \Phi(Q)$ for any $P, Q \in \Im(\inp)$.
The case $P=-\infty$ or $Q=-\infty$ is clear.
Assume $P \neq -\infty$ and $Q \neq -\infty$.
If $P(\bfit p) > Q(\bfit p)$, then
$$\Phi(P \oplus_{\bfit p} Q) = \Phi(P) = (P_{\mathbb B}, P(\bfit p)) = (P_{\mathbb B}, P(\bfit p)) + (Q_{\mathbb B}, Q(\bfit p)) = \Phi(P) + \Phi(Q).$$
The case $P(\bfit p) < Q(\bfit p)$ is similar.
If $P(\bfit p) = Q(\bfit p)$, then
$$\begin{aligned}
\Phi(P \oplus_{\bfit p} Q) &= \Phi(P \oplus Q) = ((P \oplus Q)_{\mathbb B}, (P \oplus Q)(\bfit p)) = (P_{\mathbb B}, P(\bfit p)) + (Q_{\mathbb B}, Q(\bfit p)) \\
&= \Phi(P) + \Phi(Q).
\end{aligned}$$
We next show that $\Phi(P \odot Q) = \Phi(P) \cdot \Phi(Q)$ for any $P, Q \in \Im(\inp)$.
The case $P=-\infty$ or $Q=-\infty$ is clear.
Assume $P \neq -\infty$ and $Q \neq -\infty$.
Thus
$$\Phi(P \odot Q) = ((P \odot Q)_{\mathbb B}, P(\bfit p) \odot Q(\bfit p)) = (P_{\mathbb B}, P(\bfit p)) \cdot (Q_{\mathbb B}, Q(\bfit p)) = \Phi(P) \cdot \Phi(Q).$$
Finally, $\Phi(0) = (0, 0)$.
Therefore $\Phi$ is a semiring homomorphism.\\

\noindent
Step2. $\Phi$ is bijective.

Consider the map
$$\Psi: \Bxpm \times_e \mathbb T \to \Im(\inp), \quad (P(\bfit x), a) \mapsto a \odot P(\bfit x - \bfit p), \quad -\infty \mapsto -\infty.$$
We now check that $a \odot P(\bfit x - \bfit p)$ is in $\Im(\inp)$ for any $P \in \Bxpm \smallsetminus \{ -\infty \}$ and $a \in \mathbb R$.
The all coefficients of the terms of $P$ are 0.
Thus any term of $P(\bfit x - \bfit p)$ is of the form
$((-p_1) \odot x_1)^{u_1} \odot \cdots \odot (-p_n) \odot x_n)^{u_n}.$
When we substitute $\bfit x = \bfit p$ for each terms, all the values are 0.
It means that $\inp(P(\bfit x - \bfit p)) = P(\bfit x - \bfit p)$, and then $a
 \odot P(\bfit x - \bfit p) = \inp(a \odot P(\bfit x - \bfit p)) \in \Im(\inp)$.

We show that $\Psi$ is the inverse map of $\Phi$.
It is obvious that $\Psi(\Phi(-\infty)) = -\infty$ and $\Phi(\Psi(-\infty)) = -\infty$.
For any $P \in \Im(\inp) \smallsetminus \{ -\infty \}$,
$$\Psi(\Phi(P)) = \Psi((P_{\mathbb B}, P(\bfit p))) = P(\bfit p) \odot P_{\mathbb B}(\bfit x - \bfit p) = P(\bfit x),$$
where we use Lemma \ref{inpP1}.
For any $(P, a) \in (\Bxpm \times_e \mathbb T) \smallsetminus \{ -\infty \}$,
$$\Psi(\Phi((P, a))) = \Psi(a \odot P(\bfit x - \bfit p)) = ((a \odot P(\bfit x - \bfit p))_{\mathbb B}, a \odot P({\mathbf 0})) = ((a \odot P(\bfit x - \bfit p))_{\mathbb B}, a).$$
The remaining thing to show is $(a \odot P(\bfit x - \bfit p))_{\mathbb B} = P$.
Since $P(\bfit x - \bfit p) = P(x_1 \odot (-p_1), \ldots, x_n \odot (-p_n))$, $P$ and $a \odot P(\bfit x - \bfit p)$ coincide except for coefficients.
This means that $(a \odot P(\bfit x - \bfit p))_{\mathbb B} = P$.
\end{proof}

Later we use the following lemma.

\begin{lem}
\label{lem:cong on im}
Let $\mathcal C$ be a congruence on $\Txpm$ such that $\Ker(\inp) \subset \mathcal C$.
Let $\mathcal C'$ be the congruence on $\Im(\inp)$ corresponding to $\mathcal C / \Ker(\inp)$ via the isomorphism $\Txpm / \Ker(\inp) \to \Im(\inp)$.
Then $\mathcal C' = \mathcal C \cap (\Im(\inp))^2$, and hence $\Txpm / \mathcal C$ is isomorphic to $\Im(\inp) / (\mathcal C \cap (\Im(\inp))^2)$.
\end{lem}

\begin{proof}
If $(P,Q) \in \mathcal C'$, there exist $P_0, Q_0 \in \Txpm$ such that $P = \inp(P_0)$, $Q = \inp(Q_0)$ and $(P_0,Q_0) \in \mathcal C$.
Note that $(P, P_0) \in \Ker(\inp)$ because $\inp(P) = \inp(\inp(P_0)) = \inp (P_0)$.
Hence $(P, P_0) \in \mathcal C$.
Similarly $(Q, Q_0) \in \mathcal C$.
Therefore $P \simwrt{\mathcal C} P_0 \simwrt{\mathcal C} Q_0 \simwrt{\mathcal C} Q$, which means that $(P,Q) \in \mathcal C \cap (\Im(\inp))^2$.

Conversely, if $(P,Q) \in \mathcal C \cap (\Im(\inp))^2$, then $(P,Q) \in \mathcal C'$ because $\inp(P) = P$, $\inp(Q) = Q$ and $(P,Q) \in \mathcal C$.
\end{proof}

\subsection{Local theory of tropical Laurent polynomials}

We show that $\Txpmfcn / \Cp$ is isomorphic to the $\mathbb T$-extension of $\Bxpmfcn$.
Let $\pi : \Txpm \to \Txpmfcn$ be the canonical surjection.
In order to use Proposition \ref{TtoB}, we first show that $\pi^{-1}(\Cp) \supset \Ker(\inp)$.

\begin{lem}
\label{PandinpP}
For any tropical Laurent polynomial $P \in \Txpm$, there exists an open neighborhood $U$ of $\bfit p$ such that $[P]|_U = [\inp(P)]|_U$.
Moreover, for any tropical Laurent polynomials $P_1, \ldots, P_k \in \Txpm$, there exists an open neighborhood $U$ of $\bfit p$ such that $[P_i]|_U = [\inp(P_i)]|_U$ for any $i=1, \ldots, k$.
\end{lem}

\begin{proof}
Let $P= \bigoplus_{i=1}^k a_i \odot \bfit x^{\bfit u_i}$.
We may assume that $\inp (P) = \bigoplus_{i=1}^{k'} a_i \odot \bfit x^{\bfit u_i}$ for some $k'$ without loss of generality.
Let $P' = \bigoplus_{i=k'+1}^k a_i \odot \bfit x^{\bfit u_i}$ be the sum of the terms of $P$ which do not appear in $\inp(P)$.
Thus $P=\inp(P) \oplus P'$.
By the definition of initial forms, we have $\inp(P)(\bfit p) > P'(\bfit p)$.
Since the functions $[\inp(P)]$ and $[P']$ are continuous, there exists an open neighborhood $U$ of $\bfit p$ such that $\inp(P)(\bfit q) > P'(\bfit q)$ for any $\bfit q \in U$.
Hence $[P]|_U = [\inp(P) \oplus P']|_U = [\inp(P)]|_U$.

For $P_1, \ldots, P_k \in \Txpm$, there exists open neighborhoods $U_1, \ldots, U_k$ of $\bfit p$ such that $[P_i]|_{U_i} = [\inp(P_i)]|_{U_i}$ for any $i=1, \ldots, k$.
Let $U = U_1 \cap \cdots \cap U_k$, and then $[P_i]|_U = [\inp(P_i)]|_U$ for any $i=1, \ldots, k$.
\end{proof}

\begin{prop}
Let $P,Q \in \Txpm$ be tropical Laurent polynomials such that $\inp(P) = \inp(Q)$.
Then there exists an open neighborhood $U$ of $\bfit p$ such that $[P]|_U = [Q]|_U$.
In other words, $\pi^{-1}(\Cp) \supset \Ker(\inp)$. 
\end{prop}

\begin{proof}
Let $U$ be an open neighborhood of $\bfit p$ such that $[P]|_U = [\inp(P)]|_U$ and $[Q]|_U = [\inp(Q)]|_U$.
Then
\[ [P]|_U = [\inp(P)]|_U = [\inp(Q)]|_U = [Q]|_U. \qedhere \]
\end{proof}

The semiring $\Txpmfcn / \Cp$ is isomorphic to $\Txpm / \pi^{-1}(\Cp)$, which is also isomorphic to
$$\Im(\inp) / (\pi^{-1}(\Cp) \cap (\Im(\inp))^2)$$
by Lemma \ref{lem:cong on im}.
We denote $\Cp' := \pi^{-1}(\Cp) \cap (\Im(\inp))^2$.
In other words,
$$\Cp' = \left\{ (P,Q) \in (\Im(\inp))^2 \ \middle| \ 
\begin{aligned}
&\text{ there exists an open neighborhood $U$ of $\bfit p$ } \\
&\text{ such that }  [P]|_U = [Q]|_U
\end{aligned} \right\}.$$

Since $\Im(\inp)$ is isomorphic to $\Bxpm \times_e \mathbb T$, there exists a congruence $\Cp''$ on $\Bxpm \times_e \mathbb T$ corresponding to $\Cp'$.
Recall that the isomorphism $\Bxpm \times_e \mathbb T \to \Im(\inp)$ we constructed in the proof of Proposition \ref{TtoB} was
$$(P,a) \mapsto a \odot P(\bfit x - \bfit p), \quad -\infty \mapsto -\infty.$$
Thus we obtain the following description:
$$\Cp'' = \left\{ ((P,a),(Q,b)) \ \middle| \ 
\begin{aligned}
&\text{ there exists an open neighborhood $U$ of $\bfit p$ } \\
&\text{ such that }  [a \odot P(\bfit x - \bfit p)]|_U = [b \odot Q(\bfit x - \bfit p)]|_U
\end{aligned} \right\} \cup \{ (-\infty, -\infty) \}.$$
There is a simpler description:

\begin{lem}
\label{lem:Cpdd}
We have
$$\Cp'' = \left\{ ((P, a), (Q,a)) \ \middle| \ 
[P]=[Q] \text{ and } a \in \mathbb R \right\} \cup \{ (-\infty, -\infty) \}.$$
\end{lem}

\begin{proof}
It is sufficient to show that, for any $P,Q \in \Bxpm$ and $a,b \in \mathbb R$, the following are equivalent:
\begin{enumerate}[(1)]
\item There exists an  open neighborhood $U$ of $\bfit p$ such that
$$[a \odot P(\bfit x - \bfit p)]|_U = [b \odot Q(\bfit x - \bfit p)]|_U,$$
\item $[P]=[Q]$ and $a=b$.
\end{enumerate}
Obviously (2) implies (1).
Assume (1).
Since $P$ and $Q$ have coefficients in $\mathbb B$, by substituting $\bfit x = \bfit p$ for $[a \odot P(\bfit x - \bfit p)]|_U = [b \odot Q(\bfit x - \bfit p)]|_U$, we have $a=b$, and hence
$$[P(\bfit x - \bfit p)]|_U = [Q(\bfit x - \bfit p)]|_U.$$
Thus $[P]|_{U-\bfit p} = [Q]|_{U-\bfit p}$, where $U - \bfit p := \{ \bfit q - \bfit p \ | \ \bfit q \in U \}$.
Since $U - \bfit p$ is an open neighborhood of $\mathbf 0$, by Lemma \ref{Bfcn}, we have $[P]=[Q]$.
\end{proof}

We have shown that the following semirings are isomorphic:
$$\Txpmfcn / \Cp, \quad \Txpm / \pi^{-1}(\Cp), \quad \Im(\inp) / \Cp', \text{ and } (\Bxpm \times_e \mathbb T) / \Cp''.$$
Finally, we use Lemma \ref{Textquot}.
By Lemma \ref{lem:Cpdd}, $\Cp''$ is partial, and hence $(\Bxpm \times_e \mathbb T) / \Cp''$ is isomorphic to $(\Bxpm / \CpB) \times_e \mathbb T$, where $\CpB$ be the congruence on $\Bxpm$ defined as
$$\CpB := \{ (P,Q) \in (\Bxpm \smallsetminus \{ - \infty \})^2 \ | \ ((P,0),(Q,0)) \in \Cp'' \} \cup \{ (-\infty, -\infty) \}.$$
By Lemma \ref{lem:Cpdd}, it is easily seen that
$$\CpB = \{ (P,Q) \in (\Bxpm)^2 \ | \ [P]=[Q] \},$$
which means that $\Bxpm / \CpB = \Bxpmfcn$.
Therefore we have the following result:

\begin{prop}
\label{prop:localText}
The semiring $\Txpmfcn / \Cp$ is isomorphic to $\Bxpmfcn \times_e \mathbb T$.
The isomorphism $\Txpmfcn / \Cp \to \Bxpmfcn \times_e \mathbb T$  is given by
$$(\text{\textrm{the class of }} [P]) \mapsto \begin{cases}
([(\inp(P))_{\mathbb B}], P(\bfit p)) & (P \neq -\infty)\\
-\infty & (P = -\infty),
\end{cases}$$
and the inverse map is
$$([P], a) \mapsto (\text{the class of } [a \odot P(\bfit x - \bfit p)]), \quad -\infty \mapsto -\infty.$$
\end{prop}

We conclude this section with our first main theorem.

\begin{thm}
\label{main1}
\begin{enumerate}[$(1)$]
\item Any congruence $\mathcal C$ on $\Txpmfcn$ including $\mathcal C_{\bfit p}$ satisfies just one of the following conditions:
\begin{enumerate}[$(a)$]
\item if $(f, g) \in \mathcal C$, then $f(\bfit p) = g(\bfit p)$, and
\item $(f, g) \in \mathcal C$ for any $f, g \in \Txpmfcn \smallsetminus \{ -\infty \}$.
\end{enumerate}
\item If $\mathcal C$ is a congruence on $\Txpmfcn$ including $\mathcal C_{\bfit p}$ and satisfying the condition (a), then there exists a congruence $\mathcal C'$ on $\Bxpmfcn$ such that
$\Txpmfcn / \mathcal C $ is isomorphic to $(\Bxpmfcn / \mathcal C') \times_e \mathbb T$.
\end{enumerate}
\end{thm}

\begin{proof}
\begin{enumerate}[(1)]
\item Let us denote by $\Phi : \Txpmfcn / \Cp \to \Bxpmfcn \times_e \mathbb T$ the isomorphism in the previous proposition.
Since $\mathcal C$ includes $\mathcal C_{\bfit p}$, the congruence $\mathcal C / \mathcal C_{\bfit p}$ on $\Txpmfcn / \mathcal C_{\bfit p}$ is defined. 
Let $\mathcal C'$ be the congruence on $\Bxpmfcn \times_e \mathbb T$ corresponding to $\mathcal C / \mathcal C_{\bfit p}$ via $\Phi$.
Note that, for $f, g \in \Txpmfcn$, $f \simwrt{\mathcal C} g$ if and only if
$$\Phi(\text{the class of } f) = \Phi(\text{the class of } g).$$
If $f = [P], g=[Q]$, it is equivalent to
$$f \neq -\infty, g \neq -\infty \text{ and } ([(\inp(P))_{\mathbb B}], f(\bfit p)) \simwrt{\mathcal C'} ([(\inp(Q))_{\mathbb B}], g(\bfit p)),\text{ or}$$
$$f=g=-\infty.$$
By Lemma \ref{Textcong}, $\mathcal C'$ satisfies just one of the following conditions:
\begin{enumerate}[($a'$)]
\item If $(f, a) \simwrt{\mathcal C'} (g, b)$, then $a=b$.
\item $(f, a) \simwrt{\mathcal C'} (g, b)$ for any $f, g \in \Bxpmfcn \smallsetminus \{ -\infty \}$ and $a,b \in \mathbb R$.
\end{enumerate}

We show that if $\mathcal C'$ satisfies $(a')$, then $\mathcal C$ satisfies $(a)$.
Assume that $\mathcal C'$ satisfies $(a')$.
Take $(f, g) \in \mathcal C$.
If $f = -\infty$, we have $-\infty \simwrt{\mathcal C'} \Phi( \text{the class of } g)$.
By the argument before Lemma \ref{Textcongr}, $\Phi( \text{the class of } g)= -\infty$, hence $g= -\infty$, therefore $f(\bfit p) = g(\bfit p) = -\infty$.
If $f \neq -\infty$, we have $-\infty \neq \Phi( \text{the class of } f) \simwrt{\mathcal C'} \Phi( \text{the class of } g)$.
Again by the argument before Lemma \ref{Textcongr}, $\Phi( \text{the class of } g) \neq -\infty$, hence $g \neq -\infty$.
Thus we have $([(\inp(P))_{\mathbb B}], f(\bfit p)) \simwrt{\mathcal C'} ([(\inp(Q))_{\mathbb B}], g(\bfit p))$, therefore $f(\bfit p) = g(\bfit p)$ by $(a')$.
Thus $\mathcal C$ satisfies $(a)$.

Next we show that if $\mathcal C'$ satisfies $(b')$, then $\mathcal C$ satisfies $(b)$.
Assume that $\mathcal C'$ satisfies $(b')$.
Let $f, g \in \Bxpmfcn \smallsetminus \{ -\infty \}$.
Since $\Phi(\text{the class of } f) \neq -\infty$ and $\Phi(\text{the class of } g) \neq -\infty$, $\Phi(\text{the class of } f) \simwrt{\mathcal C'} \Phi(\text{the class of } g)$ by $(b')$.
Hence $f \simwrt{\mathcal C} g$.
It means that $\mathcal C$ satisfies $(b)$.

Therefore $\mathcal C$ satisfies $(a)$ or $(b)$.
It is clear that there is no congruence on $\Bxpmfcn$ satisfying $(a)$ and $(b)$.
Thus the statement $(1)$ holds.

\item Let $\mathcal C'$ be the same one to the proof of (1).
Thus $\mathcal C'$ satisfies $(a')$ because, if not, $\mathcal C'$ satisfies $(b')$, and then $\mathcal C$ satisfies $(b)$, which contradicts the assumption that $\mathcal C$ satisfies $(a)$.
In other words, $\mathcal C'$ is partial.
We have the isomorphisms $\Txpmfcn / \mathcal C \to (\Txpmfcn / \mathcal C_p) / (\mathcal C / \mathcal C_p) \to (\Bxpmfcn \times_e \mathbb T)/\mathcal C'$.
By Lemma \ref{Textquot}, there is a congruence $\mathcal C''$ on $\Bxpmfcn$ such that $(\Bxpmfcn \times_e \mathbb T)/\mathcal C'$ is isomorphic to $(\Bxpmfcn / \mathcal C'') \times_e \mathbb T$.  \qedhere
\end{enumerate}
\end{proof}

\section{Local theory of functions on a set}

Let $X$ be a subset of $\mathbb R^n$ and $\bfit p \in X$ be a vector.
We consider the following congruence:

\begin{dfn}
We define the congruence $\CXp$ on $\Txpmfcn$ as
$$\CXp = \left\{ (f, g) \in (\Txpmfcn)^2  \ \middle| \ 
\begin{aligned}
&\text{ there exists an open neighborhood $U$ of $\bfit p$ in $\mathbb R^n$ } \\
&\text{ such that }  f|_{U \cap X} = g|_{U \cap X}
\end{aligned} \right\}.$$
\end{dfn}

Clearly $\CXp \supset \mathcal C_{\bfit p}$, and $\CXp$ satisfies the condition $(a)$ in Theorem \ref{main1} (1).
Thus, by Theorem \ref{main1} (2), there is a corresponding congruence $\CXpB$ on $\Bxpmfcn$ such that there is an isomorphism
$$\Txpmfcn / \CXp \to (\Bxpmfcn / \CXpB) \times_e \mathbb T.$$
We can describe $\CXpB$ explicitly:

\begin{lem}
\label{CXpB}
The congruence $\CXpB$ coincides with
$$\left\{ (f, g) \in \left( \Bxpmfcn \right)^2 \ \middle| \ 
\begin{aligned}
&\text{ there exists an open neighborhood $U$ of $\mathbf 0$ } \\
&\text{ such that }  f|_{U \cap (X- \bfit p)} = g|_{U \cap (X - \bfit p)}
\end{aligned} \right\},$$
where $X- \bfit p = \{ \bfit r - \bfit p \in \mathbb R^n \ | \ \bfit r \in X \}$.
\end{lem}

\begin{proof}
Let $\mathcal C'$ be the congruence on $\Bxpmfcn \times_e \mathbb T$ corresponding to $\CXp / \mathcal C_{\bfit p}$ via the isomorphism $\Txpmfcn / \mathcal C_{\bfit p} \to \Bxpmfcn \times_e \mathbb T$ in Proposition \ref{prop:localText}.
Thus, by the proof of Theorem \ref{main1} (2), $\CXpB$ is the congruence corresponding to $\mathcal C'$ via the correspondence in Lemma \ref{Textcongr}.
Hence, for $f, g \in \Bxpmfcn$,
$$(f, g) \in \CXpB \Longleftrightarrow \text{`` $f \neq -\infty, g \neq -\infty$ and $((f, 0), (g, 0)) \in \mathcal C'$ ''}  \text{ or } f = g = -\infty.$$
Recall that the isomorphism $\Bxpmfcn \times_e \mathbb T \to \Txpmfcn / \mathcal C_{\bfit p}$ is given by
$$(f, a) \mapsto (\text{the class of } a \odot f(\bfit x - \bfit p)), \quad -\infty \mapsto -\infty$$
by Proposition \ref{prop:localText}.
Then, for $f, g \in \Bxpmfcn \smallsetminus \{ -\infty \}$,
\begin{align*}
((f, 0), (g, 0)) \in \mathcal C' &\Longleftrightarrow (\text{the class of } f(\bfit x - \bfit p), \text{the class of } g(\bfit x - \bfit p)) \in \CXp / \mathcal C_{\bfit p}\\
&\Longleftrightarrow (f(\bfit x - \bfit p), g(\bfit x - \bfit p)) \in \CXp.
\end{align*}
Hence
\begin{align*}
(f, g) &\in \CXpB \\
&\Longleftrightarrow \text{`` $f \neq -\infty, g \neq -\infty$ and $(f(\bfit x - \bfit p), g(\bfit x - \bfit p)) \in \CXp$ ''}  \text{ or } f = g = -\infty \\
&\Longleftrightarrow (f(\bfit x - \bfit p), g(\bfit x - \bfit p)) \in \CXp,
\end{align*}
where the ``$\Longleftarrow$'' part of the second equivalence holds because $(f, -\infty), (-\infty, f) \not\in \CXp$ for any $f \neq -\infty$.
Moreover, we have the equivalence
\begin{align*}
(f(\bfit x &- \bfit p), g(\bfit x - \bfit p)) \in \CXp \\
&\Longleftrightarrow \text{ there exists an open neighborhood $U$ of $\bfit p$ such that }  f(\bfit q - \bfit p) = g(\bfit q - \bfit p) \\
&\text{ for any } \bfit q \in U \cap X \\
&\Longleftrightarrow \text{ there exists an open neighborhood $V$ of $\mathbf 0$ such that } f(\bfit q) = g(\bfit q) \\
&\text{ for any } \bfit q \in V \cap (X - \bfit p),
\end{align*}
where the ``$\Longrightarrow$'' part (resp. the ``$\Longleftarrow$'' part) of the second equivalence is shown by taking $V = U - \bfit p := \{ \bfit q - \bfit p \ | \ \bfit q \in U \}$ (resp. $U = V + \bfit p := \{ \bfit q + \bfit p \ | \ \bfit q \in V \}$).
This shows the lemma.
\end{proof}

In particular, when $X$ is the support of a 1-dimensional fan, we also have the following description.

\begin{lem}
Let $X$ be a 1-dimensional fan.
Then $\mathcal C \left( |X| \right)_{\mathbf 0, \mathbb B}$ coincides with
$$\mathcal C(X) := \left\{ (f, g) \in \left( \Bxpmfcn \right)^2 \ \middle| \ f|_{|X|} = g|_{|X|} \right\}.$$
\end{lem}

\begin{proof}
By Lemma \ref{CXpB},
$$\mathcal C \left( |X| \right)_{\mathbf 0, \mathbb B} = \left\{ (f, g) \in \left( \Bxpmfcn \right)^2 \ \middle| \ 
\begin{aligned}
&\text{ there exists an open neighborhood $U$ of $\mathbf 0$ } \\
&\text{ such that }  f|_{U \cap |X|} = g|_{U \cap |X|}
\end{aligned} \right\},$$
which coincides with $\mathcal C(X)$ by Lemma \ref{Bfcn}.
\end{proof}

\section{Function semirings of 1-dimensional tropical fans}
\label{sec:function}

In the rest of paper, vectors in $\mathbb R^n$ are always regarded as column vectors.
On the other hand, when we use a multi-index such as $\bfit x^{\bfit u} = x_1^{u_1} \odot \cdots \odot x_n^{u_n}$, we regard the index vector $\bfit u$ as a row vector.
We use the following notation:
For two sets $A$ and $B$, we denote $B^A$ the set of all maps from $A$ to $B$.

\subsection{Weighted evaluation maps}

Throughout this section, we fix a 1-dimensional tropical fan $X = (X, \omega_X)$ in $\mathbb R^n$.
We study the properties of the semiring $\Bxpmfcn / \CX$, and the relation between $X$ and $\Bxpmfcn / \CX$.

First, we see that the semiring $\Bxpmfcn / \CX$ is isomorphic to a subsemiring of $(\mathbb Z \cup \{ -\infty \})^{X(1)}$, where we assume that $(\mathbb Z \cup \{ -\infty \})^{X(1)}$ is equipped with the pointwise max-plus operation, denoted by $\oplus$ for addition and $\odot$ for multiplication.
Let $\bfit d_{\rho}$ be the primitive direction vector of $\rho \in X(1)$ and $w_{\rho} := \omega_X(\rho)$.
For $f, g \in \Bxpmfcn$, $(f, g) \in \CX$ if and only if $f(t\bfit d_{\rho}) = g(t\bfit d_{\rho})$ for any $\rho \in X(1)$ and $t \geq 0$.
By Lemma \ref{Bfcn}, it is equivalent to $f(\bfit d_{\rho}) = g(\bfit d_{\rho})$ for any $\rho \in X(1)$.
Consider the map
$$\phi:\Bxpmfcn \to (\mathbb Z \cup \{ -\infty \})^{X(1)}, f \mapsto \phi(f),$$
where $\phi(f)$ is the map $X(1) \to \mathbb Z \cup \{ -\infty \}, \rho \mapsto f(\bfit d_{\rho})$.
Thus it is easily checked that $\phi$ is a semiring homomorphism and $\Ker(\phi) = \CX$, and hence $\Bxpmfcn / \CX$ is isomorphic to the subsemiring $\Im (\phi)$ of $(\mathbb Z \cup \{ -\infty \})^{X(1)}$.

Instead of that $\phi$, we use the following ``weighted'' map for a reason we describe later.

\begin{dfn}
The \textit{weighted evaluation map} $\phi_X$ of $X$ is
$$\phi_X:\Bxpmfcn \to (\mathbb Z \cup \{ -\infty \})^{X(1)}, \ f \mapsto \phi_X(f),$$
where $\phi_X(f)$ is the map $X(1) \to \mathbb Z \cup \{ -\infty \}, \rho \mapsto w_{\rho} f(\bfit d_{\rho})$, and the multiplication in $w_{\rho} f(\bfit d_{\rho})$ is the standard one.
\end{dfn}

It is easily checked that $\phi_X$ is also a semiring homomorphism and $\Ker(\phi_X) = \CX$.
Thus there is an isomorphism $\Bxpmfcn / \CX \to \Im(\phi_X)$ induced by $\phi_X$.
Note that $\Im(\phi_X)$ is generated by $\{ \phi_X([x_1]), \phi_X([x_1^{-1}]), \ldots, \phi_X([x_n]), \phi_X([x_n^{-1}]) \}$ as semiring since $\Bxpmfcn$ is generated by $\{ [x_1], [x_1^{-1}], \ldots, [x_n], [x_n^{-1}] \}$.

By the following proposition, we can reconstruct $X$ from $\phi_X$.
This is the reason why we use the weighted evaluation maps.

\begin{prop}
\label{phitostar}
Let $F_i := \phi_X([x_i])$ for $i= 1, \ldots, n$.
Let $\bfit F(\rho) := {}^t \matac {F_1(\rho)} {\cdots} {F_n(\rho)}$ for any $\rho \in X(1)$.
Then $\bfit F(\rho) = w_{\rho} \bfit d_{\rho}$. 
In addition, $w_{\rho}$ is the greatest common divisor of the coordinates of $\bfit F(\rho)$.
\end{prop}

\begin{proof}
This is clear by the definition of $\phi_X$.
\end{proof}

\begin{rem}
\label{rem:reconst}
We cannot reconstruct $X$ from the semiring homomorphism $\Bxpmfcn \to \Bxpmfcn / \CX$.
Indeed, if $X' = (X', \omega_{X'})$ is another 1-dimensional tropical fan such that $|X'| = |X|$ and $\omega_{X'} \neq \omega_X$, the congruences $\CX$ and $\mathcal C(X')$ coincide by definition.
Such $X'$ always exists.
For example, the weight function $\omega_{X'} := 2 \omega_X$ works.
\end{rem}

Let $A$ be a finite set.
For any element $F \in (\mathbb Z \cup \{ -\infty \})^A$, we denote $\deg F = \bigodot_{a \in A} F(a) = \sum_{a \in A} F(a)$, which we call the \textit{degree} of $F$.
Clearly, degrees have the following properties: For any $F, G \in (\mathbb Z \cup \{ -\infty \})^A$,
\begin{itemize}
\item $\deg(F \odot G) = \deg F \odot \deg G = \deg F + \deg G$,
\item $\deg(F \oplus G) \geq \max\{ \deg F, \deg G \}, \text{ and the equality holds if and only if } F \oplus G = F \text{ or } F \oplus G = G$.
\end{itemize}

The weighted evaluation map $\phi_X$ has the following properties.

\begin{prop}
\label{imphi}
\begin{enumerate}[$(1)$]
\item For a Boolean Laurent monomial $P \in \Bxpm$, $\deg (\phi_X([P])) = 0$.
\item For any function $f \in \Bxpmfcn \smallsetminus \{ -\infty \}$, $\deg (\phi_X(f)) \geq 0$.
\end{enumerate}
\end{prop}

\begin{proof}
\begin{enumerate}[(1)]
\item We can write $P = \bfit x^{\bfit u}$ for some $\bfit u \in \mathbb Z^n$.
Then
$$\deg (\phi_X([P])) = \sum_{\rho \in X(1)} w_{\rho} P(\bfit d_{\rho}) = \sum_{\rho \in X(1)} w_{\rho} \bfit u \cdot \bfit d_{\rho} = \bfit u \cdot \sum_{\rho \in X(1)} w_{\rho} \bfit d_{\rho} = \bfit u \cdot {\mathbf 0} = 0,$$
where we use the balancing condition $\sum_{\rho \in X(1)}w_{\rho} \bfit d_{\rho} = {\mathbf 0}$.
\item We can write $f = [P_1] \oplus \cdots \oplus [P_k]$ for some Boolean Laurent monomials $P_1, \ldots, P_k$.
Hence
\[ \deg (\phi_X(f)) = \deg (\phi_X([P_1]) \oplus \cdots \oplus \phi_X([P_k])) \geq \deg (\phi_X([P_1])) = 0. \qedhere \]
\end{enumerate}
\end{proof}

Let $A$ be a finite set.
We denote
$$\mathbb Z_0^A := \left\{ F \in \mathbb Z^A \ | \ \deg F = 0 \right\}, \quad \mathbb Z_{\pos}^A := \left\{ F \in \mathbb Z^A \ | \ \deg F \geq 0 \right\}.$$
Thus $\Zpos A$ is a subsemiring of $(\mathbb Z \cup \{ -\infty \})^A$, where, by abuse of notation, $-\infty$ means the map which maps any element to $-\infty$.
Moreover, $\mathbb Z^A_0$ is the unit group of $\Zpos A$ because, if $F \in \Zpos A$ has the inverse element $G \in \Zpos A$, then
$$0 = \deg (F \odot G) = \deg F + \deg G \geq \deg F \geq 0,$$
which means that $\deg F = 0$, and hence every unit has degree $0$.
Conversely, for any element $F \in \mathbb Z^A_0$, the map $A \ni a \mapsto -F(a) \in \mathbb Z \cup \{ -\infty \}$ is the inverse element of $F$.
In addition, $\mathbb Z^A_0$ is a free abelian group of rank $|A|-1$ with respect to multiplication.

By Proposition \ref{imphi}, we may regard that the target of $\phi_X$ is $\Zpos {X(1)}$, and do so in the rest of paper.

\begin{ex}
\label{exL}
Let $r, n$ be integers such that $2 \leq r \leq n+1$.
Let $\{ \bfit e_1, \ldots, \bfit e_n \}$ be the standard basis of $\mathbb R^n$ and $\bfit e_0 := -(\bfit e_1 + \cdots + \bfit e_{r-1})$.
Let $\rho_i$ be the ray spanned by $\bfit e_i$ for $i=0,1, \ldots, r-1$, and $L_{n,r} := \{ \{ {\mathbf 0} \}, \rho_0, \rho_1, \ldots, \rho_{r-1} \}$.
Consider the 1-dimensional tropical fan $L_{n,r} = (L_{n,r}, \omega_{L_{n,r}})$, where every weight $\omega_{L_{n,r}}(\rho_i)$ is $1$.
This fan is one of the \textit{standard smooth model} defined in \cite{CDMY}.

We show that $\Im(\phi_{L_{n,r}}) = \Zpos {L_{n,r}(1)}$, i.e., $\phi_{L_{n,r}}$ is surjective.
It is enough to show that $\Zpos {L_{n,r}(1)}$ is generated by the set of the images of $[x_1], [x_1^{-1}], \ldots, [x_n], [x_n^{-1}]$.
For $i = 1, \ldots, r-1$, the image of $[x_i]$ is the map
$$\phi_{L_{n,r}}([x_i])(\rho_j) =\begin{cases}
-1 & \text{ if } j=0 \\
1 & \text{ if } j=i \\
0 & \text{ otherwise}.
\end{cases}$$
For $i = r, \ldots, n$, ${L_{n,r}}([x_i]) = 0$, i.e., ${L_{n,r}}([x_i])(\rho) = 0$ for any $\rho \in L_{n,r}(1)$. 
Let $F_i := \phi_{L_{n,r}}([x_i])$ for $i=1, \ldots, r-1$.
For any map $G \in \mathbb Z_0^{L_{n,r}(1)}$ of degree $0$, let
$$G' = G(\rho_1) F_1 + \cdots + G(\rho_{r-1}) F_{r-1} = F_1^{G(\rho_1)} \odot \cdots \odot F_{r-1}^{G(\rho_{r-1})}.$$
Clearly $G'(\rho_i) = G(\rho_i)$ for $i = 1, \ldots, r$.
Also, $G'(\rho_0) = -G(\rho_1) - \cdots -G(\rho_{r-1}) = G(\rho_0)$ since $G$ is of degree $0$.
Hence $G' = G$, which means that $\mathbb Z_0^{L_{n,r}(1)} \subset \Im(\phi_{L_{n,r}})$.
Next we show that any map in $\mathbb Z_{\mathrm {pos}}^{L_{n,r}(1)}$ is the sum of some maps in $\mathbb Z_0^{L_{n,r}(1)}$.
For any map $G \in \mathbb Z_{\mathrm {pos}}^{L_{n,r}(1)}$, let $G_0, G_1$ be the map
$$G_0(\rho_i) = \begin{cases}
G(\rho_0) - \deg G &\text{if } i=0 \\
G(\rho_i) &\text{otherwise,}
\end{cases} \quad G_1(\rho_i) = \begin{cases}
G(\rho_1) - \deg G &\text{if } i=1 \\
G(\rho_i) &\text{otherwise.}
\end{cases}$$
Then $\deg G_0 = \deg G_1 = 0$ and $G = G_0 \oplus G_1$.
Therefore $\Im(\phi_{L_{n,r}}) = \Zpos {L_{n,r}(1)}$.
\end{ex}

Weighted evaluation maps are not always surjective.
In order to give examples, we show some lemmas.

\begin{lem}
\label{zeropartsum}
Let $A$ be a finite set.
Let $F, G \in \mathbb Z^A_0$ be distinct maps of degree $0$.
Then $\deg (F \oplus G) >0$.
\end{lem}

\begin{proof}
We may assume that $F(a_0) > G(a_0)$ for some $a_0 \in A$.
Then
\[ \deg (F \oplus G) = \sum_{a \in A} \max \{ F(a), G(a) \} \geq F(a_0) + \sum_{a \in A \smallsetminus \{ a_0 \}} G(a)  > \deg G = 0. \qedhere \]
\end{proof}

\begin{lem}
\label{zeropartsum2}
Let $A$ be a finite set.
Let $F_1, \ldots, F_k \in \mathbb Z^A_0$.
If $F_1 \oplus \cdots \oplus F_k \in \mathbb Z^A_0$, then $F_1 = \cdots = F_k$.
\end{lem}

\begin{proof}
Assume that $F_i \neq F_j$ for some $i \neq j$.
Then $\deg (F_1 \oplus \cdots \oplus F_k) \geq \deg (F_i \oplus F_j) >0$ by Lemma \ref{zeropartsum}.
This is a contradiction.
\end{proof}

\begin{lem}
\label{zeropart}
Let $A$ be a finite set.
Let $R_0$ be a (multiplicative) subgroup of $\mathbb Z^A_0$ and $R$ be the subsemiring of $\Zpos A$ generated by $R_0$.
Then $R_0 = R \cap \mathbb Z^A_0$.
\end{lem}

\begin{proof}
The inclusion $R_0 \subset R \cap \mathbb Z^A_0$ is clear.
Since $R_0$ is closed with respect to multiplication, any element of $R$ is the sum of finitely many elements of $R_0$.
If $F_1 \oplus \cdots \oplus F_k \in R \cap \mathbb Z^A_0$ for some $F_1, \ldots, F_k \in R_0$, by Lemma \ref{zeropartsum2}, we have $F_1 = \cdots = F_k$.
Hence $F_1 \oplus \cdots \oplus F_k = F_1 \in R_0$, which means $R \cap \mathbb Z^A_0 \subset R_0$.
\end{proof}

\begin{ex}
\label{exY}
Let $\bfit d_1 = \matba 12, \bfit d_2 = \matba 31$ and $\bfit d_3 = \matba {-4}{-3}$.
Let $\rho_i$ be the ray spanned by $\bfit d_i$ for $i=1, 2, 3$, and $Y := \{ \{ {\mathbf 0} \}, \rho_1, \rho_2, \rho_3 \}$.
Consider the 1-dimensional tropical fan $Y = (Y, \omega_Y)$, where every weight $\omega_Y(\rho_i)$ is $1$.
We show that the map $\phi_Y : \mathbb B[x^{\pm}, y^{\pm} ]_{\mathrm{fcn}} \to \Zpos {Y(1)}$ is not surjective.
Let $R_0$ be the subgroup of $\mathbb Z_0^{Y(1)}$ generated by $\{ \phi_Y([x]), \phi_Y([y]) \}$.
Since $\Im(\phi_Y)$ is generated by $R_0$ as a semiring, by Lemma \ref{zeropart}, we have $\Im(\phi_Y) \cap \mathbb Z_0^{Y(1)} = R_0$.
Thus, in order to show that $\phi_Y$ is not surjective, it is enough to show that $R_0 \neq \mathbb Z_0^{Y(1)}$.
Consider the bijection
$$\mathbb Z_0^{Y(1)} \to \mathbb Z^2, \quad F \mapsto (F(\rho_1), F(\rho_2)).$$
Since the image of $\phi_Y([x]), \phi_Y([y])$ are $(1, 3), (2, 1)$ respectively,
it is enough to show that $\mathbb Z^2$ is not generated by $\{ (1,3), (2,1) \}$ as a (standard additive) group.
This is true because $\det \matbb 1321 = -5 \neq \pm 1$.
\end{ex}

\begin{ex}
\label{exZ}
Let $\bfit e_1 = \matba 10, \bfit e_2 = \matba 01$.
Let $\rho_1^+, \rho_1^-, \rho_2^+, \rho_2^-$ be the ray spanned by $\bfit e_1, -\bfit e_1, \bfit e_2, -\bfit e_2$ respectively, and $Z := \{ \{ {\mathbf 0} \}, \rho_1^+, \rho_1^-, \rho_2^+, \rho_2^- \}$.
Thus the 1-dimensional tropical fan $Z = (Z, \omega_Z)$ is defined, where every weight is $1$.
We show that $\phi_Z : \mathbb B[x^{\pm}, y^{\pm} ]_{\mathrm{fcn}} \to \Zpos {Z(1)}$ is not surjective.
Let $R_0$ be the subgroup of $\mathbb Z_0^{Z(1)}$ generated by $\{ \phi_Z([x]), \phi_Z([y]) \}$.
Similar to previous example, it is enough to show that $R_0 \neq \mathbb Z_0^{Z(1)}$.
This is true because $\mathbb Z_0^{Z(1)}$ is a free abelian group of rank $3$, and hence not generated by two elements.
\end{ex}

Let $A$ be a finite set.
We sometimes use column vectors with entries in $\ZApos$.
When we denote $\bfit F := {}^t \! \matac {F_1} {\cdots} {F_n}$, the notation $\bfit F(a)$ means the column vector ${}^t \! \matac {F_1(a)} {\cdots} {F_n(a)}$.
For a homomorphism $\phi : \Bxpmfcn \to \Zpos A$, we denote $\bfit F_{\phi} := {}^t\!\matac{\phi([x_1])}{\cdots}{\phi([x_n])}$.
We can reconstruct $\phi$ by $\bfit F_{\phi}$ because $\phi$ is determined by the images of $[x_1], \ldots, [x_n]$.

\subsection{Realizable homomorphisms}
\label{subsec:RealHom}
What kind of homomorphism can be realized as $\phi_X$ for some 1-dimensional tropical fan $X$?
In this section, we give an answer for this question.

Throughout this section, we fix finite sets $A, B$.

\begin{dfn}
Let $\sigma : A \to B$ be a map.
Then we define the semiring homomorphism $\sigma^* : \Zpos B \to \Zpos A$ as $\sigma^* (F) = F \circ \sigma$.
\end{dfn}

If $\sigma : A \to B$ is bijection, obviously $\sigma^*$ is an isomorphism.

\begin{dfn}
Two homomorphisms $\phi_1 : \Bxpmfcn \to \Zpos A$ and $\phi_2 : \Bxpmfcn \to \Zpos B$ are \textit{equivalent} if there exists a bijection $\sigma : A \to B$ such that $\phi_1 = \sigma^* \circ \phi_2$.
\end{dfn}

Clearly, this gives an equivalence relation.

\begin{dfn}
A homomorphism $\phi:\Bxpmfcn \to \Zpos A$ is \textit{realizable} if $\phi$ is equivalent to $\phi_X$ for some 1-dimensional tropical fan $X$.
\end{dfn}

\begin{lem}
\label{xtodeg0}
Let $\phi:\Bxpmfcn \to \Zpos A$ be an arbitrary homomorphism.
Then $\phi([x_i]) \in \mathbb Z^A_0$ for any $i = 1, \ldots, n$.
\end{lem}

\begin{proof}
This is because $[x_i]$ is invertible and $\mathbb Z^A_0$ is the unit group of $\Zpos A$.
\end{proof}

\begin{prop}
\label{propreal}
Let $\phi:\Bxpmfcn \to \Zpos A$ be an arbitrary homomorphism.
Then the following are equivalent:
\begin{enumerate}[(1)]
\item $\phi$ is realizable.
\item $\bfit F_{\phi}(a) \neq {\mathbf 0}$ for any $a \in A$, and any $\bfit F_{\phi}(a)$ is not the scalar multiple of any $\bfit F_{\phi}(a')$ ($a' \neq a$) by a positive real number.
\end{enumerate}
\end{prop}

\begin{proof}
$(1) \Longrightarrow (2)$: This immediately follows from Proposition \ref{phitostar}.

$(2) \Longrightarrow (1)$: For any $a \in A$, we can uniquely write $\bfit F_{\phi}(a) = w_a \bfit d_a$ for some positive integer $w_a$ and primitive vector $\bfit d_a$.
Let $\rho_a$ be the ray spanned by $\bfit d_a$ and $X := \{ \rho_a \ | \ a \in A \} \cup \{ \{ \mathbf 0 \}\}$.
Note that $\rho_a \neq \rho_{a'}$ if $a \neq a'$ by the assumption $(2)$.
Consider the weighted 1-dimensional fan $X = (X, \omega_X)$, where $\omega_X(\rho_a) = w_a$ for any $a \in A$.
We show that $X$ is a 1-dimensional tropical fan.
By Lemma \ref{xtodeg0}, $\sum_{a \in A}\bfit F_{\phi}(a) = \mathbf 0$, hence $\sum_{a \in A} w_a \bfit d_a = \mathbf 0$.
Therefore, the balancing condition holds, and then $X$ is a 1-dimensional tropical fan.
The map $\sigma : A \to X(1), a \mapsto \rho_a$ is a bijection, and for any $a \in A$,
\begin{align*}
\bfit F_{\sigma^* \circ \phi_X}(a) &= {}^t\!\matac{(\sigma^* \circ \phi_X)([x_1])(a)}{\cdots}{(\sigma^* \circ \phi_X)([x_n])(a)} \\
&= {}^t\!\matac{(\sigma^*(\phi_X([x_1])))(a)}{\cdots}{(\sigma^*(\phi_X([x_n])))(a)} \\
&= {}^t\!\matac{\phi_X([x_1])(\rho_a)}{\cdots}{\phi_X([x_n])(\rho_a)} = w_a \bfit d_a = \bfit F_{\phi}(a),
\end{align*}
which means that $\bfit F_{\sigma^* \circ \phi_X} = \bfit F_{\phi}$.
Since $\phi$ can be reconstructed by $\bfit F_{\phi}$, we have $\sigma^* \circ \phi_X = \phi$.
Hence $\phi$ is realizable.
\end{proof}

\begin{prop}
Let $\phi:\Bxpmfcn \to \Zpos A$ be a realizable homomorphism.
Then the 1-dimensional tropical fan $X$ such that $\phi$ is equivalent to $\phi_X$ is unique.
\end{prop}

\begin{proof}
Let $X = (X, \omega_X)$ be a 1-dimensional tropical fan such that $\phi$ is equivalent to $\phi_X$.
Let $\sigma : A \to X(1)$ be a bijection such that $\sigma^* \circ \phi_X = \phi$.
Thus, for any $a \in A$, $\bfit F_{\phi}(a) = \bfit F_{\phi_X}(\sigma(a))$.
By Proposition \ref{phitostar}, the ray $\sigma(a)$ and its weight are determined by the vector $\bfit F_{\phi_X}(\sigma(a))$, and hence by $\bfit F_{\phi}(a)$.
Thus $\phi$ determines $(X, \omega_X)$.
\end{proof}

The fan $X$ in the above proposition is called \textit{the fan defined by} $\phi$.
Clearly, two realizable homomorphisms define the same fan if and only if they are equivalent.
Thus we have the following result.

\begin{prop}
The correspondence $X \mapsto \phi_X$ gives a bijection from the set of 1-dimensional tropical fans in $\mathbb R^n$ to the set of the equivalence classes of semiring homomorphisms $\phi:\Bxpmfcn \to \Zpos A$ for some finite set $A$ satisfying the following condition:
$\bfit F_{\phi}(a) \neq {\mathbf 0}$ for any $a \in A$, and any $\bfit F_{\phi}(a)$ is not the scalar multiple of any $\bfit F_{\phi}(a')$ ($a' \neq a$) by a positive real number.
\end{prop}

Intuitively, for a 1-dimensional tropical fan $X \subset \mathbb R^n$, one may consider that the semirings $\Bxpmfcn$ and $\Im(\phi_X)$ correspond to $\mathbb R^n$ and $X$ respectively, and the homomorphism $\phi_X : \Bxpmfcn \to \Im(\phi_X)$ corresponds to  the inclusion map $X \hookrightarrow \mathbb R^n$.
This observation suggests that there is an ``abstract 1-dimensional tropical fan'' defined by the semiring $\Im(\phi_X)$.
We do not define it in this paper because we have not known the precise definition yet.

The following lemma is used in Section \ref{sec:Morphisms}.

\begin{lem}
\label{unitgroup}
Let $R_0$ be a (multiplicative) subgroup of $\mathbb Z^A_0$ and $R$ be the subsemiring of $\Zpos A$ generated by $R_0$.
Then the unit group of $R$ is $R_0 = R \cap \mathbb Z^A_0$.
In particular, for any homomorphism $\phi : \Bxpmfcn \to \Zpos A$, the unit group of $\Im(\phi)$ is $\Im(\phi) \cap \mathbb Z^A_0$.
\end{lem}

\begin{proof}
The equality $R_0 = R \cap \mathbb Z^A_0$ follows from Lemma \ref{zeropart}.
Let $U$ be the unit group of $R$.
Since every element of $R_0$ is invertible, $R_0 \subset U$.
Since any unit of $R$ has degree $0$, $U \subset R \cap \mathbb Z^A_0$.
Thus $R_0 \subset U \subset R \cap \mathbb Z^A_0 = R_0$, which means that $U = R_0 = R \cap \mathbb Z^A_0$.

For any homomorphism $\phi : \Bxpmfcn \to \Zpos A$, the image $\Im(\phi)$ is generated by $\{ \phi([x_1]), \phi([x_1])^{-1}, \ldots, \phi([x_n]), \phi([x_n])^{-1} \} \subset \mathbb Z^A_0$.
Thus $\Im(\phi)$ is generated by the subgroup of $\mathbb Z^A_0$ generated by $\{ \phi([x_1]), \ldots, \phi([x_n]) \}$.
Hence the latter statement follows from the former.
\end{proof}

\subsection{Images of realizable homomorphisms}

We study properties of the images of realizable homomorphisms.
We fix a finite set $A$.

\begin{dfn}
A subsemiring $R$ of $\ZApos$ is \textit{realizable} if $R = \Im(\phi)$ for some realizable homomorphism $\phi$. 
\end{dfn}

We give a condition for a subsemiring of $\ZApos$ to be realizable.
We use the following lemma.

\begin{lem}
\label{universal}
For any elements $F_1, \ldots, F_n \in \mathbb Z^A_0$, there exists a unique semiring homomorphism $\phi : \Bxpmfcn \to \ZApos$ such that $\phi([x_i]) = F_i$ for any $i=1, \ldots, n$.
\end{lem}

\begin{proof}
Since $\Bxpmfcn$ is generated by $\{ x_1, x_1^{-1}, \ldots, x_n, x_n^{-1} \}$, such $\phi$ is unique if it exists.
We show the existence.
Since $F_1, \ldots, F_n$ are invertible, there exists a unique semiring homomorphism $\psi :  \Bxpm \to \ZApos$ such that $\psi(x_i) = F_i$ for $i=1, \ldots, n$.
We show that $\psi$ induces a homomorphism $\Bxpmfcn \to \ZApos$, i.e., $\psi(P) = \psi(Q)$ if $[P] = [Q]$.

Let $P, Q \in \Bxpm$ be Boolean Laurent polynomials such that $[P]=[Q]$.
Since $\psi$ is a semiring homomorphism,
$$\psi(P(x_1, \ldots, x_n))(a) = P(\psi(x_1)(a), \ldots, \psi(x_n)(a)).$$
Similarly, $\psi(Q(x_1, \ldots, x_n))(a) = Q(\psi(x_1)(a), \ldots, \psi(x_n)(a))$.
Since $[P] = [Q]$, we have $\psi(P)(a) = \psi(Q)(a)$, which means that $\psi(P)=\psi(Q)$.
\end{proof}

\begin{prop}
\label{realiff}
A subsemiring $R$ of $\ZApos$ is realizable if and only if there exist finitely many elements $F_1, \ldots, F_n \in R \cap \mathbb Z^A_0$ satisfying the following conditions:
\begin{enumerate}[$(1)$]
\item The set $\{ F_1, F_1^{-1} \ldots, F_n, F_n^{-1} \}$ generates $R$ as a semiring.
\item Let $\bfit F = {}^t\!\matac{F_1}{\cdots}{F_n}$.
Then $\bfit F(a) \neq {\mathbf 0}$ for any $a \in A$, and any $\bfit F(a)$ is not the scalar multiple of any $\bfit F(a')$ ($a' \neq a$) by a positive real number.
\end{enumerate}
\end{prop}

\begin{proof}
Assume that $R$ is realizable.
Let $\phi : \Bxpmfcn \to \ZApos$ be a realizable homomorphism such that $\Im(\phi) = R$.
Let $F_i := \phi([x_i])$ for $i=1, \ldots, n$.
Then $F_1, \ldots, F_n \in R \cap \mathbb Z^A_0$ by Lemma \ref{xtodeg0}.
These $F_1, \ldots, F_n$ satisfy the condition $(1)$ because $\Bxpmfcn$ is generated by $\{ x_1, x_1^{-1}, \ldots, x_n, x_n^{-1} \}$, and $(2)$ by Proposition \ref{propreal}.

Conversely, assume that there are $F_1, \ldots, F_n \in R \cap \mathbb Z^A_0$ satisfying $(1), (2)$.
By Lemma \ref{universal}, there exists a semiring homomorphism $\phi : \Bxpmfcn \to \ZApos$ such that $\phi([x_i]) = F_i$ for any $i$.
Thus $\phi$ is a realizable homomorphism by Proposition \ref{propreal} and the condition $(2)$.
The image $\Im(\phi)$ is $R$ by the condition $(1)$.
Therefore, $R$ is realizable.
\end{proof}

The conditions we gave in the above proposition do not depend on the choice of generators:

\begin{prop}
\label{notdepend}
Let $R \subset \ZApos$ be a realizable subsemiring.
Let $F_1, \ldots, F_n \in R \cap \mathbb Z^A_0$ be elements such that $R$ is generated by $\{ F_1, F_1^{-1} \ldots, F_n, F_n^{-1} \}$ as a semiring.
Let $\bfit F := {}^t\!\matac{F_1}{\cdots}{F_n}$.
Then $\bfit F(a) \neq {\mathbf 0}$ for any $a \in A$, and any $\bfit F(a)$ is not the scalar multiple of any $\bfit F(a')$ ($a' \neq a$) by a positive real number.
\end{prop}

\begin{proof}
Since $R$ is realizable, there exist maps $G_1, \ldots, G_m \in R \cap \mathbb Z^A_0$ satisfying the following conditions:
\begin{enumerate}[(1)]
\item The set $\{ G_1, G_1^{-1} \ldots, G_m, G_m^{-1} \}$ generates $R$ as a semiring.
\item Let $\bfit G := {}^t\!\matac{G_1}{\cdots}{G_m}$.
Then $\bfit G(a) \neq {\mathbf 0}$ for any $a \in A$, and any $\bfit G(a)$ is not the scalar multiple of any $\bfit G(a')$ ($a' \neq a$) by a positive real number.
\end{enumerate}
Let $R_0$ be the subgroup of $\mathbb Z^A_0$ generated by $\{ F_1, \ldots, F_n \}$.
By Lemma \ref{zeropart}, $R_0 = R \cap \mathbb Z^A_0$.
Hence $G_1, \ldots, G_m \in R_0$.
Thus we can write $G_j = F_1^{t_{1j}} \odot \cdots \odot F_n^{t_{nj}} = t_{1j} F_1 + \cdots + t_{nj} F_n$ for some $t_{ij} \in \mathbb Z$.
This means that there is an $n \times m$ integer matrix $T$ such that $\bfit G = T \bfit F$, where the multiplication in the right hand side is the standard matrix multiplication.
This means $\bfit G(a) = T \bfit F(a)$ for any $a \in A$.
If some $\bfit F(a)$ is $\mathbf 0$, then $\bfit G(a) = T \bfit F(a) = \mathbf 0$, which contradicts the condition $(2)$.
Similarly, if $\bfit F(a) = t \bfit F(a')$ for some $a \neq a'$ and $t > 0$, then $\bfit G(a) = t \bfit G(a')$, which contradicts the condition $(2)$.
\end{proof}

\begin{rem}
Let $R$ be a realizable subsemiring of $\ZApos$.
Giving a surjective homomorphism $\Bxpmfcn \to R$ is equivalent to fixing $n$ elements $F_1, \ldots, F_n \in R \cap Z^A_0$ such that $R$ is generated by $\{ F_1, F_1^{-1} \ldots, F_n, F_n^{-1} \}$ as a semiring.
Thus, intuitively, ``not depend on the choice of generators'' is equivalent to ``not depend on the choice of embedding to $\mathbb R^n$''.
This observation suggests the existence of abstract 1-dimensional tropical fan defined by a realizable subsemiring of $\ZApos$.
\end{rem}

There are more things that do not depend on the choice of generators.
If $t_1 \bfit d_1 + \cdots + t_r \bfit d_r = \mathbf 0$ for $t_1, \ldots, t_r \in \mathbb 
R$ and $\bfit d_1, \ldots, \bfit d_r \in \mathbb R^n$, we say $(t_1, \ldots, t_r)$ is a \textit{linear relation} of $\bfit d_1, \ldots, \bfit d_r$.

\begin{prop}
Let $R \subset \ZApos$ be a realizable subsemiring.
Let $F_1, \ldots, F_n \in R \cap \mathbb Z^A_0$ be elements such that $R$ is generated by $\{ F_1, F_1^{-1} \ldots, F_n, F_n^{-1} \}$ as a semiring.
Let $\bfit F := {}^t\!\matac{F_1}{\cdots}{F_n}$.
Then the linear relations of the vectors $\left( \bfit F(a) \right)_{a \in A}$ do not depend on the choice of $F_1, \ldots, F_n$.
\end{prop}

\begin{proof}
Let $G_1, \ldots, G_m \in R \cap \mathbb Z^A_0$ be elements such that $R$ is generated by the set $\{ G_1, G_1^{-1}, \ldots, G_m, G_m^{-1} \}$.
Let $\bfit G := {}^t\!\matac {G_1}{\cdots}{G_m}$.
By the same argument as the proof of Proposition \ref{notdepend}, there exists an integer matrices $T$ such that $\bfit G = T \bfit F$.
If $(t_a)_{a \in A}$ is a linear relation of $\left( \bfit F(a) \right)_{a \in A}$, then
$$\sum_{a \in A} t_a \bfit F(a) = \mathbf 0.$$
Thus
$$\sum_{a \in A} t_a \bfit G(a) = \sum_{a \in A} t_a T \bfit F(a) = T \sum_{a \in A} t_a \bfit F(a) = \mathbf 0,$$
hence $(t_a)_{a \in A}$ is a linear relation of $\left( \bfit G(a) \right)_{a \in A}$.
By the similar argument, any linear relation of $\left( \bfit G(a) \right)_{a \in A}$ is a linear relation of $\left( \bfit F(a) \right)_{a \in A}$.
\end{proof}

This proposition suggests that the abstract 1-dimensional tropical fan defined by a realizable subsemiring of $\ZApos$ should have some ``linear'' structure.

\section{Morphisms}
\label{sec:Morphisms}

In the rest of paper, we fix a positive integer $m$.
The notation $\Bypm$ always means $\mathbb B[y_1^{\pm}, \ldots, y_m^{\pm}]$.

\subsection{Morphisms of realizable homomorphisms}
\label{subsec:MorDef}

In Section \ref{sec:function}, we study the map $X \to \phi_X$ from the set of 1-dimensional tropical fans to the set of realizable homomorphisms.
We make it into a functor.
To do this, we need to define the morphisms of realizable homomorphisms.

We may describe the definition of morphisms of 1-dimensional tropical fans as follows: 
Let $X \subset \mathbb R^n$ and $Y \subset \mathbb R^m$ be 1-dimensional tropical fans.
A morphism $\mu : X \to Y$ is a map $\mu : |X| \to |Y|$ such that there exists a linear map $\mu_0 : \mathbb R^n \to \mathbb R^m$ such that the diagram
$$\xymatrix{
|X| \ar@{^{(}->}[r] \ar[d]_{\mu} & \mathbb R^n \ar[d]_{\mu_0} \\
|Y| \ar@{^{(}->}[r] & \mathbb R^m
}$$
is commutative.
Thus, it is natural that we define the morphisms of realizable homomorphisms as follows:
Let $\phi : \Bxpmfcn \to \ZApos$ and $\psi : \Bypmfcn \to \ZBpos$ be realizable homomorphisms.
A morphism $\nu : \phi \to \psi$ is a semiring homomorphism $\nu : \Im(\phi) \to \Im(\psi)$ such that there exists a semiring homomorphism $\nu_0 : \Bxpmfcn \to \Bypmfcn$ such that the diagram
$$\xymatrix{
\Bxpmfcn \ar[r]^{\phi} \ar[d]_{\nu_0} & \Im(\phi) \ar[d]_{\nu} \\
\Bypmfcn \ar[r]^{\psi} & \Im(\psi)
}$$
is commutative.
However, by the following lemma, we do not need to consider the existence of a homomorphism $\Bxpmfcn \to \Bypmfcn$.

\begin{lem}
Let $A, B$ be finite sets.
Let $\phi : \Bxpmfcn \to \ZApos$ and $\psi : \Bypmfcn \to \ZBpos$ be semiring homomorphisms.
For any semiring homomorphism $\nu : \Im(\phi) \to \Im(\psi)$, there exists a semiring homomorphism $\nu_0 : \Bxpmfcn \to \Bypmfcn$ such that the diagram
$$\xymatrix{
\Bxpmfcn \ar[r]^{\phi} \ar[d]_{\nu_0} & \Im(\phi) \ar[d]_{\nu} \\
\Bypmfcn \ar[r]^{\psi} & \Im(\psi)
}$$
is commutative.
\end{lem}

\begin{proof}
Note that the unit groups of $\Im(\phi)$ and $\Im(\psi)$ are $\Im(\phi) \cap \mathbb Z^A_0$ and $\Im(\psi) \cap \mathbb Z^B_0$ respectively by Lemma \ref{unitgroup}.
Thus $\nu$ induces the group homomorphism $\Im(\phi) \cap \mathbb Z^A_0 \to \Im(\psi) \cap \mathbb Z^B_0$.

Let $F_i = \phi([x_i])$ and $G_i = \psi([y_i])$.
Then $F_i \in \Im(\phi) \cap \mathbb Z^A_0$ and $G_i \in \Im(\psi) \cap \mathbb Z^B_0$ by Lemma \ref{xtodeg0}.
Let $S_0$ be the subgroup of $\mathbb Z^B_0$ generated by $\{ G_1, \ldots, G_m \}$.
By Lemma \ref{zeropart}, $S_0 = \Im(\psi) \cap \mathbb Z^B_0$.
Thus any element of $\Im(\psi) \cap \mathbb Z^B_0$ is of the form $G_1^{t_1} \odot \cdots \odot G_m^{t_m}$ for some $t_1, \ldots, t_m \in \mathbb Z$.
In particular,
$$\nu(F_i) = G_1^{t_{i1}} \odot \cdots \odot G_m^{t_{im}}$$
for some $t_{ij} \in \mathbb Z$.

By Lemma \ref{universal1}, there exists a unique homomorphism $\nu_0 : \Bxpmfcn \to \Bypmfcn$ satisfying $\nu_0([x_i]) = [y_1^{t_{i1}} \odot \cdots \odot y_m^{t_{im}}]$.
We show that this $\nu_0$ makes the diagram commutative.
It is enough to show that $(\psi \circ \nu_0)([x_i]) = (\nu \circ \phi)([x_i])$ for any $i$.
This is true because
$$(\psi \circ \nu_0)([x_i]) = \psi([y_1^{t_{i1}} \odot \cdots \odot y_m^{t_{im}}]) = G_1^{t_{i1}} \odot \cdots \odot G_m^{t_{im}},$$
while
\[ (\nu \circ \phi)([x_i]) = \nu(F_i) = G_1^{t_{i1}} \odot \cdots \odot G_m^{t_{im}}. \qedhere \]
\end{proof}

Thus we define the morphisms of realizable homomorphisms as follows.

\begin{dfn}
Let $\phi : \Bxpmfcn \to \ZApos$ and $\psi : \Bypmfcn \to \ZBpos$ be realizable homomorphisms.
A \textit{morphism} $\nu : \phi \to \psi$ of realizable homomorphisms is a semiring homomorphism $\nu : \Im(\phi) \to \Im(\psi)$.
\end{dfn}

\subsection{Pullbacks of functions}

Let $\mu : \mathbb R^n \to \mathbb R^m$ be a linear map defined by an integer matrix $T$.
For a function $f \in \Bypmfcn$ on $\mathbb R^m$, the composition $f \circ \mu$ is a function on $\mathbb R^n$.
We see that $f \circ \mu \in \Bxpmfcn$, i.e. $f \circ \mu$ is the function defined by a Boolean Laurent polynomial.

Let $\bfit t_j$ be the $j$-th row of $T$.
Then
$$\mu (\bfit p) = T \bfit p = {}^t\!\matac {\bfit t_1 \bfit p}{\cdots}{\bfit t_m \bfit p}$$
for any $\bfit p \in \mathbb R^n$.
Thus, if $f$ is the function defined by the Boolean Laurent polynomial $P \in \Bypm$,
$$(f \circ \mu) (\bfit p) = P(\bfit t_1 \bfit p, \ldots, \bfit t_m \bfit p).$$
Hence $f \circ \mu$ is the function defined by a Boolean Laurent polynomial $P(\bfit x^{\bfit t_1}, \ldots, \bfit x^{\bfit t_m})$.

Thus the map $\mu^* : \Bypmfcn \to \Bxpmfcn, f \mapsto f \circ \mu$ is defined, which is called the \textit{pullback along} $\mu$.
The map $\mu^*$ is clearly a semiring homomorphism.

Next, let $X \subset \mathbb R^n$ and $Y \subset \mathbb R^m$ be 1-dimensional tropical fans and $\mu:X \to Y$ be a morphism.
Thus $\mu$ is induced by a linear map $\mu_0 : \mathbb R^n \to \mathbb R^m$.
For a function $f \in \Bypmfcn / \CY$ on $|Y|$, the composition $f \circ \mu$ is a function on $|X|$, which is defined by a Boolean Laurent polynomial by the above argument.
Hence the map
$$\mu^* : \Bypmfcn / \CY \to \Bxpmfcn / \CX, \quad f \mapsto f \circ \mu$$
is defined, which is also called the \textit{pullback along} $\mu$.
This $\mu^*$ is also a semiring homomorphism.
We have the commutative diagram
$$\xymatrix{
\Bypmfcn \ar[r] \ar[d]_{\mu_0^*} & \Bypmfcn / \CY \ar[d]_{\mu^*} \\
\Bxpmfcn \ar[r] & \Bxpmfcn / \CX .
}$$

Since $\Bypmfcn / \CY$ and $\Bxpmfcn / \CX$ are isomorphic to $\Im(\phi_Y)$ and $\Im(\phi_X)$ respectively, $\mu^*$ induces a homomorphism $\Im(\phi_Y) \to \Im(\phi_X)$.
We denote it by $\mu^*$ by abuse of notation.

\begin{prop}
\label{contra}
The correspondence $X \mapsto \phi_X$ and $\mu \mapsto \mu^*$ gives a contravariant functor from the category of 1-dimensional tropical fans to the category of realizable homomorphisms.
\end{prop}

\begin{proof}
It is clear that the pullbacks along identity morphisms are identity morphisms.
For two morphisms $\mu_1 : X \to Y$ and $\mu_2 : Y \to Z$ of 1-dimensional tropical fans and a function $f$ on $Z$, we have
$$(\mu_2 \circ \mu_1)^*(f) = f \circ (\mu_2 \circ \mu_1) = (f \circ \mu_2) \circ \mu_1 = \mu_1^*(\mu_2^*(f)),$$
which means that $(\mu_2 \circ \mu_1)^* = \mu_1^* \circ \mu_2^*$.
\end{proof}

We show that this functor is faithful.
We use the following notation:
Let $A, B$ be finite sets and $\nu : \ZApos \to \ZBpos$ be a semiring homomorphism.
Let $F_1, \ldots, F_n \in \mathbb Z^A_0$ be any elements of degree $0$ and $\bfit F := {}^t\!\matac{F_1}{\cdots}{F_n}$.
Then $\nu(\bfit F)$ means the column vector ${}^t\!\matac{\nu(F_1)}{\cdots}{\nu(F_n)}$.

We use the following lemma.

\begin{lem}
\label{mu and mu*}
Let $X \subset \mathbb R^n$ and $Y \subset \mathbb R^m$ be 1-dimensional tropical fans and $\mu : X \to Y$ be a morphism.
Let $\rho \in X$ be a ray, $w$ be its weight, and $\bfit d$ be its primitive direction vector.
\begin{enumerate}[$(1)$]
\item For any function $f \in \Bypmfcn$ on $\mathbb R^m$, we have
$$\mu^*(\phi_Y(f))(\rho) = w f(\mu(\bfit d)).$$
\item Let $G_i := \phi_Y([y_i])$ and $\bfit G = {}^t\! \matac {G_1} {\cdots} {G_m}$. 
Then $\mu^*(\bfit G)(\rho) = w \mu(\bfit d)$.
\end{enumerate}
\end{lem}

\begin{proof}
\begin{enumerate}[(1)]
\item Let $\mu_0 : \mathbb R^n \to \mathbb R^m$ be a linear map such that $\mu_0|_{|X|} = \mu$.
Note that there is a commutative diagram
$$\xymatrix{
\Bypmfcn \ar[r] \ar@/^{20pt}/[rr]^{\phi_Y} \ar[d]_{\mu_0^*} & \Bypmfcn / \CY \ar[r] \ar[d]_{\mu^*} & \ZYpos  \ar[d]_{\mu^*} \\
\Bxpmfcn \ar[r] \ar@/_{20pt}/[rr]_{\phi_X} & \Bxpmfcn / \CX \ar[r] & \ZXpos.
}$$
Thus we have
$$\mu^*(\phi_Y(f))(\rho) = \phi_X(\mu_0^*(f))(\rho) = \phi_X(f \circ \mu_0)(\rho) = w f(\mu_0 (\bfit d)) = w f(\mu (\bfit d)).$$
\item By $(1)$, we have
\begin{align*}
\mu^*(\bfit G)(\rho) &= {}^t\! \matac {\mu^*(\phi_Y([y_1]))(\rho)} {\cdots} {\mu^*(\phi_Y([y_m]))(\rho)} \\
&={}^t\! \matac {w[y_1](\mu(\bfit d))} {\cdots} {w[y_m](\mu(\bfit d))} = w \mu(\bfit d). \qedhere
\end{align*}
\end{enumerate}
\end{proof}

\begin{prop}
\label{faithful}
Let $X \subset \mathbb R^n$ and $Y \subset \mathbb R^m$ be 1-dimensional tropical fans.
Let $\mu_1, \mu_2 : X \to Y$ be morphisms.
If $\mu_1^* = \mu_2^*$, then $\mu_1 = \mu_2$.
In other words, the functor in Theorem \ref{contra} is faithful.
\end{prop}

\begin{proof}
Let $\rho$ be a ray in $X$, $w$ be its weight, and $\bfit d$ be its primitive direction vector.
It is enough to show that $\mu_1(\bfit d) = \mu_2(\bfit d)$. 
Let $G_i := \phi_Y([y_i])$ for $i = 1, \ldots, m,$ and $\bfit G := {}^t\! \matac {G_1} {\cdots} {G_m}$.
Then, by Lemma \ref{mu and mu*},
$$w \mu_1(\bfit d) = \mu_1^*(\bfit G)(\rho) = \mu_2^*(\bfit G)(\rho) = w \mu_2(\bfit d).$$
Since $w$ is a positive integer, we have $\mu_1(\bfit d) = \mu_2(\bfit d)$.
\end{proof}

\subsection{Geometricity}

Pullback homomorphisms have the following property.

\begin{prop}
\label{pbgeometric}
Let $X \subset \mathbb R^n$ and $Y \subset \mathbb R^m$ be 1-dimensional tropical fans and $\mu : X \to Y$ be a morphism.
Let $G_i = \phi_Y([y_i])$ for $i=1, \ldots, m$, and $\bfit G := {}^t\! \matac {G_1} {\cdots} {G_m}$.
Then, for any $\rho \in X(1)$, there exists a ray $\rho' \in Y(1)$ and a nonnegative rational number $t$ such that $\mu^*(\bfit G)(\rho) = t\bfit G(\rho')$.
\end{prop}

\begin{proof}
Fix a ray $\rho \in X(1)$.
Let $w$ be the weight of $\rho$ and $\bfit d$ be the primitive direction vector of $\rho$.
By Lemma \ref{mu and mu*}, we have $\mu^*(\bfit G)(\rho) = w\mu(\bfit d) \in |Y|$.
Thus $\mu^*(\bfit G)(\rho) \in \rho'$ for some $\rho' \in Y(1)$.
Since $\bfit G(\rho') \in \rho' \smallsetminus \{ \mathbf 0 \}$ by Lemma \ref{phitostar}, we have $\mu^*(\bfit G)(\rho) = t\bfit G(\rho')$ for some $t \geq 0$.
That $t$ is rational because both $\mu^*(\bfit G)(\rho)$ and $\bfit G(\rho')$ are integer vectors.
\end{proof}

We introduce the following terminologies.

\begin{dfn}
\label{def:geometric}
Let $A, B$ be finite sets. 
\begin{enumerate}[(1)]
\item Let $R \subset \ZApos$ and $S \subset \ZBpos$ be subsemirings and $F_1, \ldots, F_n \in R \cap \mathbb Z^r_0$ be any elements of degree $0$.
Let $\bfit F := (F_1, \ldots, F_n)$.
A homomorphism $\nu : R \to S$ is \textit{geometric with respect to} $F_1, \ldots, F_n$ if, for any $b \in B$, there exist $a \in A$ and a nonnegative rational number $t$ such that $\nu(\bfit F)(b) = t\bfit F(a)$.
\item For two realizable homomorphisms $\phi_1 : \Bxpmfcn \to \ZApos$ and $\phi_2 : \Bypmfcn \to \ZBpos$, a morphism $\nu : \phi_1 \to \phi_2$ is \textit{geometric} if $\nu$ is geometric with respect to $\phi_1([x_1]), \ldots, \phi([x_n])$.
\end{enumerate}
\end{dfn}

Proposition \ref{pbgeometric} means that, for any morphism $\mu : X \to Y$ of 1-dimensional tropical fans, $\mu^* : \phi_Y \to \phi_X$ is geometric.

Conversely, we can construct a morphism of 1-dimensional tropical fans from a geometric morphism.
This is the reason why we use the term ``geometric''.

\begin{prop}
\label{full}
Let $X \subset \mathbb R^n$ and $Y \subset \mathbb R^m$ be 1-dimensional tropical fans.
Assume that there is a geometric morphism $\nu:\phi_Y \to \phi_X$.
Then there exists a unique morphism $\nu^* : X \to Y$ such that $(\nu^*)^* = \nu$.
\end{prop}

\begin{proof}
The uniqueness follows from Proposition \ref{faithful}.
We show the existence.

Let $R_X := \Im(\phi_X)$ and $R_Y := \Im(\phi_Y)$.
Note that the unit groups of $R_X$ and $R_Y$ are $R_X \cap \mathbb Z^{X(1)}_0$ and $R_Y \cap \mathbb Z^{Y(1)}_0$ respectively by Lemma \ref{unitgroup}.
Thus $\nu$ induces the group homomorphism $R_Y \cap \mathbb Z^{Y(1)}_0 \to R_X \cap \mathbb Z^{X(1)}_0$.

Let $F_i := \phi_X([x_i])$ and $\bfit F := {}^t\! \matac {F_1} {\cdots} {F_n}$.
Similarly, let $G_i := \phi_Y([y_i])$ and $\bfit G:= {}^t\! \matac {G_1} {\cdots} {G_m}$.
Since $R_X \cap \mathbb Z^{X(1)}_0$ is generated by $\{ F_1, \ldots, F_n \}$, we can write
$$\nu(G_i) = t_{i1} F_1 + \cdots + t_{in} F_n$$
for some $t_{ij} \in \mathbb Z$, for $i=1, \ldots, m$.
This means that there exists an $m \times n$ integer matrix $T$ such that $\nu(\bfit G) = T \bfit F$.

Let $\bfit d$ be the primitive direction vector of a ray $\rho \in X(1)$ and $w$ be the weight of $\rho$.
Then $w \bfit d = \bfit F(\rho)$ by Proposition \ref{phitostar}.
Since $\nu$ is geometric with respect to $G_1, \ldots, G_m$, there exists a ray $\rho' \in Y(1)$ and a nonnegative rational number $t$ such that $\nu(\bfit G)(\rho) = t \bfit G(\rho')$.
Thus we have
$$t \bfit G(\rho') = \nu(\bfit G)(\rho) = T\bfit F (\rho) = w T \bfit d.$$
Since $\bfit G (\rho') \in \rho'$, we have $T \bfit d \in \rho' \subset |Y|$.
Thus $T$ defines a map from $|X|$ to $|Y|$.
In other words, $T$ defines a morphism from $X$ to $Y$.
We denote this morphism by $\nu^*$.

We show $(\nu^*)^* = \nu$.
It is enough to show that $(\nu^*)^*(G_i) = \nu(G_i)$ for any $i$ since $R_Y$ is generated by $\{ G_1, G^{-1}, \ldots, G_m, G_m^{-1} \}$.
Thus, it is enough to show that $(\nu^*)^*(\bfit G) = \bfit G$.
Let $\rho$ be a ray in $X$, $w$ be its weight, and $\bfit d$ be its primitive direction vector.
Then, by Lemma \ref{mu and mu*},
$$(\nu^*)^*(\bfit G)(\rho) = w \nu^*(\bfit d) = w T \bfit d.$$
On the other hand, by lemma \ref{phitostar}, $\nu(\bfit G)(\rho) = T \bfit F(\rho) =  w T \bfit d$.
Hence $(\nu^*)^* = \nu$.
\end{proof}

Hence we have the following result.

\begin{thm}
The correspondence $X \mapsto \phi_X$ and $\mu \mapsto \mu^*$ gives a fully faithful contravariant functor from the category of 1-dimensional tropical fans to the category of realizable homomorphisms whose morphisms are geometric morphisms.
\end{thm}

\begin{rem}
\label{map of rays}
Let $X, Y, \nu, \bfit F$ and $\bfit G$ be as in Proposition \ref{full} and its proof.
Assume that $\nu(\bfit G)(\rho) \neq \mathbf 0$ for any $\rho \in X(1)$.
Since $\nu$ is geometric, for any $\rho \in X(1)$, $\nu (\bfit G)(\rho) = t \bfit G(\rho')$ for some ray $\rho' \in Y(1)$ and a rational number $t>0$.
Such $\rho'$ is unique since $\bfit G(\rho') \in \rho'$ by lemma \ref{phitostar}.
Thus the correspondence $\rho \mapsto \rho'$ defines a map $X(1) \to Y(1)$.
We denote this map by $\sigma$.

We see that $\sigma$ is compatible with $\nu^*$, i.e., for any ray $\rho \in X(1)$ and nonzero vector $\bfit p \in \rho$, we have $\nu^* (\bfit p) \in \sigma(\rho)$.
Since $\bfit F(\rho)$ is nonzero vector in $\rho$, it is enough to show that $\nu^*(\bfit F(\rho)) \in \sigma(\rho)$.
Let $T$ be as in the proof of Proposition \ref{full}.
Then, by the proof of Proposition \ref{full},
$$\nu^*(\bfit F(\rho)) = T \bfit F(\rho) \in \rho' = \sigma(\rho).$$

In particular, if $\sigma$ is injective (resp. surjective, resp. bijective), then $\nu^*$ is also injective (resp. surjective, resp. bijective).
\end{rem}

\subsection{Geometric homomorphisms}

The author has the following conjecture, which means that the term ``geometric'' is not necessary.

\begin{conj}
\label{conj}
Let $A, B$ be finite sets.
Let $R \subset \ZApos$ and $S \subset \ZBpos$ be subsemirings.
Then any homomorphism $\nu : R \to S$ is geometric with respect to any set of elements of degree $0$.
\end{conj}

In the situation above, let $\iota : S \hookrightarrow \ZBpos$ be the inclusion map.
Then it is clear that $\nu$ is geometric with respect to a set of elements of degree $0$ if and only if $\iota \circ \nu$ is so.
Thus, in order to prove this conjecture, it is enough to consider the homomorphisms of the form $R \to \ZBpos$.

We give some simple examples.

\begin{prop}
\label{prop:inclusion}
Let $A$ be a finite set.
Let $R \subset \ZApos$ be a subsemiring.
Then the inclusion map $\iota : R \hookrightarrow \ZApos$ is geometric with respect to any set of elements of degree $0$.
\end{prop}

\begin{proof}
Clear by definition.
\end{proof}

\begin{prop}
\label{prop:pullback is geom}
Let $A$ and $B$ be finite sets such that $|A| = |B|$.
Let $\sigma : A \to B$ be a bijection.
Then the morphism $\sigma^* : \ZBpos \to \ZApos, F \mapsto F \circ \sigma$ is geometric with respect to any set of elements of degree $0$.
\end{prop}

\begin{proof}
Clear by definition.
\end{proof}

The property ``geometric with respect to a generating set of the unit group'' is remarkable.

\begin{prop}
Let $A, B$ be finite sets.
Let $R \subset \ZApos$ be a subsemiring.
Assume that $R$ is generated by a subgroup $R_0$ of $\mathbb Z^A_0$.
Let $\nu:R \to \ZBpos$ be a geometric homomorphism with respect to a generating set of $R_0$.
Then $\nu$ is geometric with respect to any elements of $R_0$.
\end{prop}

\begin{proof}
Let $\{ F_1, \ldots, F_n \}$ be a generating set of $R_0$ such that $\nu$ is geometric with respect to them.
Let $G_1, \ldots, G_m$ be any elements of $R_0$.
We show that $\nu$ is geometric with respect to $G_1, \ldots, G_m$.

Let $\bfit F := {}^t\! \matac {F_1} {\cdots} {F_n}$ and $\bfit G = {}^t\! \matac {G_1} {\cdots} {G_m}$.
Since $\{ F_1 ,\ldots, F_n \}$ generates $R_0$, there exists an integer matrix $T$ such that $\bfit G = T \bfit F$.
Thus $\nu(\bfit G) = \nu(T \bfit F) = T \nu(\bfit F)$, where the second equality holds by the following reason:
Let $T = (t_{ij})_{i,j}$.
Then the $i$-th component of $T \bfit F$ is
$$\matac {t_{i1}}{\cdots}{t_{in}} \matca {F_1} {\vdots} {F_n} = t_{i1}F_1 + \cdots + t_{in}F_n = F_1^{t_{i1}} \odot \cdots, \odot F_n^{t_{in}}.$$
Since $\nu$ is a semiring homomorphism, the $i$-th component of $\nu(T \bfit F)$ is
$$\nu(F_1^{t_{i1}} \odot \cdots \odot F_n^{t_{in}}) = \nu(F_1)^{t_{i1}} \odot \cdots \odot \nu(F_n)^{t_{in}} = \matac {t_{i1}}{\cdots}{t_{in}} \matca {\nu(F_1)} {\vdots} {\nu(F_n)},$$
which is the $i$-th component of $T \nu(\bfit F)$.

Since $\nu$ is geometric with respect to $F_1, \ldots, F_n$, for any $b \in B$, there exist some $a \in A$ and some rational number $t \geq 0$ such that $\nu(\bfit F)(b) = t \bfit F(a)$.
Hence $\nu(\bfit G)(b) = T \nu(\bfit F)(b) = t T \bfit F(a) = t \bfit G(b)$.
Therefore $\nu$ is geometric with respect to $G_1, \ldots, G_m$.
\end{proof}

\begin{cor}
Let $A,B, R, R_0$ and $\nu$ be as in the previous proposition.
Then the following are equivalent:
\begin{enumerate}[$(1)$]
\item $\nu$ is geometric with respect to any set of elements in $R_0$,
\item $\nu$ is geometric with respect to a generating set of $R_0$.
\end{enumerate}
\end{cor}

We say $\nu$ is \textit{geometric} if $\nu$ satisfies the condition $(1)$ or $(2)$, and hence  $(1)$ and $(2)$.
This terminology is compatible with Definition \ref{def:geometric} (2).

We give some examples of geometric homomorphisms.
We use the following lemma.

\begin{lem}
\label{lem:trop eq}
Let $n \geq 2$ be an integer.
Let $a_1, \ldots, a_n ,c \in \mathbb T$ be any elements.
Assume that $a_i \oplus a_j = c$ for any distinct $i, j$.
Then there exists $i$ such that $a_i \leq c$ and $a_j = c$ for any $j \in \{ 1, \ldots, n \} \smallsetminus \{ i \}$.
\end{lem}

\begin{proof}
If $a_i = c$ for any $i$, there are no things to say.
If not, we may assume that $a_1 \neq c$ without loss of generality.
For any $j=2, \ldots, n$, $a_1 \oplus a_j = c$ by the assumption.
It means that $a_1 < c$ and $a_j = c$. 
\end{proof}

\begin{prop}
Let $A, B$ be finite sets.
Assume that $|A| \geq 2$.
Then any semiring homomorphism from $\ZApos$ to $\ZBpos$ is geometric.
\end{prop}

\begin{proof}
Let $A = \{ a_0, a_1, \ldots, a_n \}$.
For any $i = 1, \ldots, n$, let $F_i \in \mathbb Z^A_0$ be the map
$$F_i(a_j) = \begin{cases}
-1 &\text{ if } j=i \\
1 &\text{ if } j=0 \\
0 &\text{ otherwise}.\\
\end{cases}$$
Thus $\{ F_1, \ldots, F_n \}$ is a generating set of $\mathbb Z^A_0$.
Let $\bfit F := {}^t\! \matac {F_1} {\cdots} {F_n}$.
Note that $\bfit F(a_0) = {}^t\! \matad 11{\cdots}1$ and
$$\bfit F(a_j) = {}^t (0 \ \cdots \ 0 \ \underset{\underset{j\text{-th}}{\uparrow}}{-1} \ 0 \ \cdots \ 0)$$
for any $j=1, \ldots, n$.

Let $\nu : \ZApos \to \ZBpos$ be a homomorphism.
We show that $\nu$ is geometric with respect to $F_1, \ldots, F_n$.
Let $F \in \ZApos$ be the map
$$F(a_j) := \begin{cases}
1 &\text{ if } j=0 \\
0 &\text{ otherwise.}
\end{cases}$$
It is easily checked that $F_i \oplus F_{i'} = F_i \oplus 0 = F$ for any distinct $i, i'$.
Since $\nu$ is a semiring homomorphism, $\nu(F_i) \oplus \nu(F_{i'}) = \nu(F_i) \oplus 0 = \nu(F)$ for any distinct $i, i'$.

Fix $b \in B$.
Then $\nu(F_i)(b) \oplus \nu(F_{i'})(b) = \nu(F_i)(b) \oplus 0 = \nu(F)(b)$ for any distinct $i, i'$.
Hence, by Lemma \ref{lem:trop eq}, one of the following conditions holds:
\begin{enumerate}[(1)]
\item $\nu(F_1)(b) = \nu(F_2)(b) = \cdots = \nu(F_n)(b) = \nu(F)(b) \geq 0$,
\item there exists $i \in \{ 1, \ldots, n\}$ such that $\nu(F_i)(b) \leq 0$ and $\nu(F_{i'})(b) = 0 = \nu(F)(b)$ for any $i' \in \{ 1, \ldots, n \} \smallsetminus \{ i \}$.
\end{enumerate}
In the case $(1)$, $\nu(\bfit F)(b) = t \bfit F(a_0)$, where $t = \nu(F)(b) \geq 0$.
In the case $(2)$, $\nu(\bfit F)(b) = t \bfit F(a_i)$, where $t = -\nu(F_i)(b) \geq 0$.
Therefore $\nu$ is geometric.
\end{proof}

In the next example, we use the following notation:
Let $A$ be a finite set.
Let $F_1, F_2 \in \mathbb Z^A_0$ be any maps of degree $0$.
Let $\bfit F := \matba {F_1} {F_2}$.
For any vector $\bfit a = \matba {a_1} {a_2} \in \mathbb Z^2$, we denote
$$\det \matab {\bfit F}{\bfit a} := \det \matbb {F_1}{a_1}{F_2}{a_2} := a_2F_1 - a_1F_2 = F_1^{a_2} \odot F_2^{-a_1},$$
which is in $\mathbb Z^A_0$.
Note that $\det \matab {\bfit F}{\bfit a}$ is a Boolean Laurent monomial of $F_1, F_2$.
Hence, if a subsemiring $R \subset \ZApos$ includes $F_1$ and $F_2$, for any semiring $S$ and a homomorphism $\nu: R \to S$, we have $\nu \left( \det \matab {\bfit F}{\bfit a} \right) = \det \matab {\nu (\bfit F)}{\bfit a}$.

\begin{prop}
Let $A=\{ 1, 2, 3 \}$, and $B$ be a finite set.
Let $F_1, F_2 \in \mathbb Z^A_0$ be two distinct maps of degree $0$.
Let $R_0$ be the subgroup of $\mathbb Z^A_0$ generated by $\{ F_1, F_2 \}$, and $R$ be the subsemiring of $\ZApos$ generated by $R_0$.
Assume that $R_0$ is a free abelian group of rank $2$.
Then any semiring homomorphism from $R$ to $\ZBpos$ is geometric.
\end{prop}

\begin{proof}
Let $\bfit F := \matba {F_1} {F_2}$.
Note that $\bfit F(1) + \bfit F(2) + \bfit F(3) = \matba 00$ since $F_1$ and $F_2$ have degree $0$.
Let $d := \det \matab {\bfit F(1)}{\bfit F (2)}$.
Then we have
$$\det \matab {\bfit F(1)}{\bfit F(3)} = \det \matab {\bfit F(1)}{-\bfit F(1) - \bfit F(2)} = -d$$
and
$$\det \matab {\bfit F(2)}{\bfit F(3)} = \det \matab {\bfit F(2)}{-\bfit F(1) - \bfit F(2)} = d.$$
Note that $d \neq 0$ since $R_0$ has rank $2$.
We may assume that $d>0$ without loss of generality (if not, exchange $F_1$ and $F_2$).
Also note that $\bfit F(i) \neq \mathbf 0$ for any $i \in A$ since $R_0$ has rank $2$.

Let $G_1 := \det \matab {\bfit F}{-\bfit F(1)}$, and $G_2 := \det \matab {\bfit F}{-\bfit F(1) - \bfit F(2)}$, which are in $\mathbb Z^A_0$.
For any $i \in A$,
$$G_1(i) = \det \matab {\bfit F(i)}{-\bfit F(1)} = \begin{cases}
0 &\text{ if } i=1 \\
d &\text{ if } i=2 \\
-d &\text{ if } i=3, \\
\end{cases}$$
and similarly,
$$G_2(i) = \det \matab {\bfit F(i)}{-\bfit F(1) - \bfit F(2)} = \begin{cases}
-d &\text{ if } i=1 \\
d &\text{ if } i=2 \\
0 &\text{ if } i=3. \\
\end{cases}$$
It is easily checked that $G_1 \oplus G_2 = G_1 \oplus 0 = G_2 \oplus 0$.
Since $\nu$ is a semiring homomorphism, we have $\nu(G_1) \oplus \nu(G_2) = \nu(G_1) \oplus 0 = \nu(G_2) \oplus 0$.
Thus, for any $b \in B$, by Lemma \ref{lem:trop eq}, one of the following conditions holds:
\begin{enumerate}[(1)]
\item $\nu(G_1)(b) = \nu(G_2)(b) \geq 0$,
\item $\nu(G_1)(b) = 0 \geq \nu(G_2)(b)$,
\item $\nu(G_2)(b) = 0 \geq \nu(G_1)(b)$.
\end{enumerate}
In the case $(1)$, we have $\det \matab {\nu (\bfit F)(b)}{-\bfit F(1)}$ = $\det \matab {\nu (\bfit F)(b)}{-\bfit F(1)-\bfit F(2)}$.
Hence $\det \matab {\nu (\bfit F)(b)}{\bfit F(2)} = 0$, which means that $\nu(\bfit F)(b) = t\bfit F(2)$ for some $t \in \mathbb R$.
Then
$$ 0 \leq \det \matab {\nu (\bfit F)(b)}{-\bfit F(1)} = \det \matab {t \bfit F(2)}{-\bfit F(1)} = td,$$
hence $t \geq 0$.
In the case $(2)$, we have $\det \matab {\nu (\bfit F)(b)}{-\bfit F(1)} = 0$, which means that $\nu(\bfit F)(b) = t\bfit F(1)$ for some $t \in \mathbb R$.
Then
$$ 0 \geq \det \matab {\nu (\bfit F)(b)}{-\bfit F(1)-\bfit F(2)} = \det \matab {t \bfit F(1)}{-\bfit F(1)-\bfit F(2)} = -td,$$
hence $t \geq 0$.
In the case $(3)$, we have $\det \matab {\nu (\bfit F)(b)}{-\bfit F(1)-\bfit F(2)} = 0$, which means that $\nu(\bfit F)(b) = t(-\bfit F(1)-\bfit F(2)) = t\bfit F(3) $ for some $t \in \mathbb R$.
Then
$$ 0 \geq \det \matab {\nu (\bfit F)(b)}{-\bfit F(1)} = \det \matab {t \bfit F(3)}{-\bfit F(1)} = -td,$$
hence $t \geq 0$.
\end{proof}

\section{Smoothness}

As an application of the results so far, we discuss about the smoothness of a 1-dimensional tropical fan at $\mathbf 0$.
We refer \cite{CDMY} for the definition of smoothness.
However, we slightly change the definition since they do not consider weights.

\begin{dfn}
\label{def:smooth}
Let $X=(X,\omega_X)$ be a 1-dimensional tropical fan in $\mathbb R^n$ with $r$ rays.
Then, $X$ is \textit{smooth at} $\mathbf 0$ if the following conditions are satisfied:
\begin{enumerate}[(1)]
\item There exists a matrix $T \in GL(n, \mathbb Z)$ such that the linear map defined by $T$ induces a bijection from $|L_{n, r}|$ to $|X|$, where $L_{n,r}$ is the 1-dimensional tropical fan introduced in Example \ref{exL}, and
\item $\omega_X(\rho) = 1$ for all $\rho \in X(1)$.
\end{enumerate}
\end{dfn}

We give a condition for a 1-dimensional tropical fan to be smooth at $\mathbf 0$ in terms of weighed evaluation maps.
We use the following lemmas.
In the proof of the first one, we use a property of integer matrices, which is shown in appendix.

\begin{lem}
\label{lem:smooth equiv}
Let $X, Y \subset \mathbb R^n$ be 1-dimensional tropical fans.
If there exist morphisms $\mu_1 : X \to Y$ and $\mu_2 : Y \to X$ such that $\mu_1 \circ \mu_2 = \id$ and $\mu_2 \circ \mu_1 = \id$, then $\mu_1$ and $\mu_2$ are defined by matrices in $GL(n, \mathbb Z)$.
\end{lem}

\begin{proof}
By symmetricity, it is enough to show that $\mu_1$ is defined by a matrix in $GL(n, \mathbb Z)$.
Let $T_1, T_2$ be $n \times n$ matrices which define $\mu_1$ and $\mu_2$ respectively.
Let $\{ \bfit d_1, \ldots, \bfit d_r \}$ be the set of primitive direction vectors of rays of $X$.
Let $\bfit d_i' := \mu_1(\bfit d_i)$ for any $i$.
Then $T_1 \bfit d_i = \bfit d_i'$ and $T_2 \bfit d_i' = \bfit d_i$ for any $i$.
Let $D := \matac {\bfit d_1}{\cdots}{\bfit d_r}$ and $D' := \matac {\bfit d_1'}{\cdots}{\bfit d_r'}$.
Then $T_1D = D'$ and $T_2D' = D$.
Hence, by Lemma \ref{lem:append1}, there exists an integer matrix $T \in GL(n, \mathbb Z)$ such that $TD=D'$.
Hence $T \bfit d_i = \bfit d_i'$ for any $i$, which means that $T$ defines the morphism $\mu_1$.
\end{proof}

\begin{lem}
\label{lem:primitive}
Let $X \subset \mathbb R^n$ be a 1-dimensional tropical fan.
Suppose that $X$ is smooth at $\mathbf 0$.
Let $T \in GL(n, \mathbb Z)$ be an integer matrix which induces a bijection from $|L_{n, r}|$ to $|X|$.
Let $\mu : L_{n, r} \to X$ be the morphism defined by $T$.
Then, for every primitive direction vector $\bfit d$ of a ray in $L_{n,r}$, the image $\mu(\bm d)$ is the primitive direction vector of a ray in $X$.
\end{lem}

\begin{proof}
It is enough to show that $\mu(\bfit d)$ is primitive.
Note that $\mu(\bfit d)$ is not $\mathbf 0$ because $T^{-1} \mu (\bfit d) = \bfit d \neq \mathbf 0$.
Thus we can uniquely write as $\mu(\bfit d) = w \bfit d'$, where $w$ is a positive integer and $\bfit d'$ is primitive.
Then $\bfit d = T^{-1} \mu(\bfit d) = w T^{-1} \bfit d'$.
Since $\bfit d$ is primitive, we have $w=1$, which means that $\mu(\bfit d)$ is primitive.
\end{proof}

\begin{thm}
Let $X=(X, \omega_X)$ be a 1-dimensional tropical fan in $\mathbb R^n$.
Then, $X$ is smooth at $\mathbf 0$ if and only if the weighted evaluation map $\phi_X$ is surjective.
\end{thm}

\begin{proof}
Suppose that $X$ is smooth at $\mathbf 0$.
Then there exists an integer matrix $T \in GL(n, \mathbb Z)$ which induces a bijection from $|L_{n, r}|$ to $|X|$.
Let $\mu : L_{n, r} \to X$  be the morphism defined by $T$.
In particular, $T$ induces a bijection $\sigma : L_{n,r}(1) \to X(1)$.
Take any function $F \in \ZXpos$.
Then $F \circ \sigma \in \Zpos {L_{n,r}(1)}$.
Since $\phi_{L_{n,r}} : \Bypmfcn \to \Zpos {L_{n,r}(1)}$ is surjective by Example \ref{exL}, there exists a function $f \in \Bypmfcn$ such that $\phi_{L_{n,r}} (f) = F \circ \sigma$.
We show that $\phi_X(f \circ \mu^{-1}) = F$, where $\mu^{-1}$ exists because $T^{-1}$ defines it.
Since every ray in $X$ is of the form $\sigma(\rho)$ for some ray $\rho \in L_{n,r}(1)$, it is enough to show that $\phi_X(f \circ \mu^{-1})(\sigma(\rho)) = F(\sigma(\rho))$ for any $\rho \in L_{n,r}(1)$.
By Lemma \ref{lem:primitive}, the primitive direction vector of $\sigma(\rho)$ is $\mu(\bfit d)$, where $\bfit d$ is the primitive direction vector of $\rho$.
Thus we have
$$\phi_X(f \circ \mu^{-1})(\sigma(\rho)) = (f \circ \mu^{-1})(\mu(\bfit d)) = f(\bfit d) = \phi_{L_{n,r}} (f)(\rho) = (F \circ \sigma)(\rho) = F(\sigma(\rho)),$$
where we use the assumption that the weight of $\sigma(\rho)$ is $1$ in the first equality.
Hence $\phi_X(f \circ \mu^{-1}) = F$, which means that $\phi_X$ is surjective.

Conversely, suppose that $\phi_X$ is surjective.
Let $r := |X(1)|$.
Fix a bijection $\sigma : L_{n,r}(1) \to X(1)$.
Then $\sigma$ induces the isomorphism $\sigma^* : \ZXpos \to \Zpos {L_{n,r}(1)}, F \mapsto F \circ \sigma$.
The inverse map $\sigma^{-1} : X(1) \to L_{n, r}(1)$ exists, and it induces the inverse morphism $(\sigma^*)^{-1} = (\sigma^{-1})^* : \Zpos {L_{n,r}(1)} \to \ZXpos$.
Since $\sigma^*$ and $(\sigma^{-1})^*$ are geometric by proposition \ref{prop:pullback is geom}, there exist unique morphisms $\mu_1 : L_{n,r} \to X$ and $\mu_2 : X \to L_{n,r}$ such that $\mu_1^* = \sigma^*$ and $\mu_2^* = (\sigma^{-1})^*$ by Proposition \ref{full}.
Since $(\mu_1 \circ \mu_2)^* = (\sigma^{-1})^* \circ \sigma^* = \id$ and $(\mu_2 \circ \mu_1)^* = \sigma^* \circ (\sigma^{-1})^* = \id$, we have $\mu_1 \circ \mu_2 = \id$ and $\mu_2 \circ \mu_1 = \id$ by Proposition \ref{faithful}.
Therefore, by Lemma \ref{lem:smooth equiv}, $\mu_1$ is defined by a matrix in $GL(n, \mathbb Z)$.
It means that $X$ satisfies the condition $(1)$ in Definition \ref{def:smooth}.

Assume that $\omega_X(\rho) \geq 2$ for some $\rho \in X(1)$.
For any function $f \in \Bxpmfcn$, $\phi_X(f)(\rho)$ is divided by $\omega_X(\rho)$ by the definition of weighted evaluation maps.
It contradicts to the surjectivity of $\phi_X$.
Hence $\omega_X(\rho)=1$ for all $\rho \in X(1)$.
\end{proof}

By this theorem, the 1-dimensional tropical fans $Y$ in Example \ref{exY} and $Z$ in Example \ref{exZ} are not smooth at $\mathbf 0$.

Let $X$ be a 1-dimensional tropical fan in $\mathbb R^n$ with $r$ rays which is not smooth at $\mathbf 0$.
The inclusion map $\Im(\phi_X) \hookrightarrow \Zpos {X(1)}$ is geometric by Proposition \ref{prop:inclusion}.
Fix an integer $m \geq r-1$ and a bijection $\sigma : L_{m,r}(1) \to X(1)$.
Then $\sigma$ induces the isomorphism $\sigma^* : \Zpos {X(1)} \to \Zpos {L_{m,r}(1)}$ which is geometric by Proposition \ref{prop:pullback is geom}.
Hence $\sigma^* \circ \iota :\Im(\phi_X) \to \Zpos {L_{m,r}(1)}$ is geometric.
Thus there exists a corresponding morphism $\mu : L_{m,r} \to X$ with respect to the correspondence in Proposition \ref{full}.
One may consider that $\mu$ is a ``resolution of singularity''.
Note that $\mu$ is bijective by the following proposition.

\begin{prop}
In the above situation, $\mu$ is bijective.
\end{prop}

\begin{proof}
By Remark \ref{map of rays}, it is enough to show that the induced map $L_{m,r}(1) \to X(1)$ is bijective.
Let $F_i := \phi_X([x_i])$ and $\bfit F := {}^t\! \matac {F_1}{\cdots}{F_n}$.
For any $\rho \in L_{m,r}(1)$, we have $(\sigma^* \circ \iota)(\bfit F)(\rho) = \bfit F(\sigma (\rho)),$ which is in $\sigma(\rho)$ by Lemma \ref{phitostar}.
Hence the induced map $L_{m,r}(1) \to X(1)$ is same as $\sigma$, which is bijective.
\end{proof}

\setcounter{section}{0}
\renewcommand{\thesection}{\Alph{section}}
\section{Integer matrices}

In this section, we prove the property of integer matrices we used in the proof of Lemma \ref{lem:smooth equiv}.
Through out this section, $E_n$ means the $n \times n$ identity matrix.

\begin{prop and dfn}
Let $A$ be an $m \times n$ integer matrix.
Then there exist matrices $P \in GL(m, \mathbb Z)$ and $Q \in GL(n, \mathbb Z)$ such that $PAQ$ is of the form $\matbb DOOO$, where
$D=\matcc {\alpha_1}{}{O}{}{\ddots}{}{O}{}{\alpha_r}$
is a diagonal matrix such that each $\alpha_i$ is a positive integer and $\alpha_i|\alpha_{i+1}$ for each $i$.
The matrix $\matbb DOOO$ is called the \textit{Smith normal form} of $A$, which is uniquely determined by $A$.
\end{prop and dfn}

\begin{proof}
See \cite[Section 5.3]{AW}.
\end{proof}

\begin{lem}
\label{lem:append3}
Let $m,n$ be positive integers such that $m \geq n$.
Let $A$ be an $m \times n$ integer matrices.
Suppose that there exists an $n \times m$ integer matrix $B$ such that $BA = E_n$.
Then there exists an integer matrix $A' \in GL(m, \mathbb Z)$ whose first $n$ columns coincide with $A$.
\end{lem}

\begin{proof}
We first show that the Smith normal form of $A$ is $\matba {E_n}O$.
Let
$$PAQ = \matba DO, \qquad D=\begin{pmatrix}
\alpha_1 &&&&& \\
& \ddots &&&O& \\
&& \alpha _r &&& \\
&&& 0 && \\
&O&&& \ddots & \\
&&&&& 0
\end{pmatrix}$$
be the Smith normal form of $A$, where $\alpha_i \neq 0$ for any $i$.
Since $BA = E_n$, we have
$$Q=BP^{-1}\matba DO.$$
Let $B'$ be the square submatrix of $BP^{-1}$ which consists of the first $n$ columns of $BP^{-1}$.
Then $Q=B'D$.
Since $Q \in GL(n, \mathbb Z)$, we have $\det D = 1$, which means that $r=n$ and $\alpha_1=\cdots=\alpha_n = 1$.
Hence the Smith normal form of $A$ is $\matba {E_n}O$.
Note that $A = P^{-1} \matba {E_n}O Q^{-1} = P^{-1} \matba {Q^{-1}}O$.

Let $A' := P^{-1} \matbb {Q^{-1}}OO{E_{m-n}}$.
Then $A' \in GL(m, \mathbb Z)$ and the first $n$ columns of $A'$ coincide with $P^{-1} \matba {Q^{-1}}O = A$.
\end{proof}

\begin{lem}
\label{lem:append2}
Let $m,n$ be positive integers such that $m \geq n$.
Let $A$ and $B$ be $m \times n$ integer matrices.
Suppose that there exist $n \times m$ integer matrices $C, D$ such that $CA = DB = E_n$.
Then there exists a matrix $T \in GL(m, \mathbb Z)$ such that $TA = B$.
\end{lem}

\begin{proof}
By Lemma \ref{lem:append3}, there exists a matrix $A' \in GL(m, \mathbb Z)$ (resp. $B' \in GL(m, \mathbb Z)$) such that the first $n$ columns of $A'$ (resp. $B'$) coincide with $A$ (resp. $B$).
Let $T :=B'A'^{-1}$, then $TA' = B'$.
Comparing the first $n$ columns of both of the sides, we have $TA=B$.
\end{proof}

\begin{lem}
\label{lem:append1}
Let $A$ and $B$ be $m \times n$ integer matrices.
Assume that there exist integer matrices $M_{AB}$ and $M_{BA}$ such that $M_{AB}A = B$ and $M_{BA}B = A$.
Then there exists an integer matrix $T \in GL(m, \mathbb Z)$ such that $TA = B$.
\end{lem}

\begin{proof}
Let $H_A$ (resp. $H_B$) be the subgroup of $\mathbb Z^n$ generated by the rows of $A$ (resp. $B$).
The equality $M_{AB}A = B$ means $H_A \supset H_B$.
Hence we have $H_A = H_B$.
Let $C$ be a matrix such that the set of its rows forms a basis of $H_A$.
Thus there exist integer matrices $M_{AC}$ and $M_{CA}$ such that $M_{AC}A = C$ and $M_{CA}C = A$.
Also, there exist integer matrices $M_{BC}$ and $M_{CB}$ such that $M_{BC}B = C$ and $M_{CB}C = B$.
Then $M_{BC}M_{AB}M_{CA}C = C$.
Since the rows of $C$ are linearly independent on $\mathbb R$, we have $M_{BC}M_{AB}M_{CA} = E_r$, where $r$ is the number of rows of $C$, i.e., the rank of $H_A$.
Similarly, we have $M_{AC}M_{BA}M_{CB} = E_r$.
Note that both $M_{CA}$ and $M_{CB}$ are $m \times r$ matrices and $r = \rank H_A \leq m$.
Hence, by Lemma \ref{lem:append2}, there exists $T \in GL(m, \mathbb Z)$ such that $TM_{CA} = M_{CB}$.
Therefore $TA = TM_{CA}C = M_{CB}C = B$.
\end{proof}

\noindent
Department of Mathematical sciences\\
Tokyo Metropolitan University\\
1-1 Minami-Ohsawa, Hachioji-shi \\
Tokyo, 192-0397, Japan\\
E-mail: ito-t@tmu.ac.jp


\begin{thebibliography}{99}
\bibitem{AR} L. Allermann and J. Rau, First steps in tropical intersection theory, Mathematische Zeitschrift, 264(3), pp. 633--670, 2010.
\bibitem{AW} W. A. Adkins and S. H. Weintraub, Algebra: an approach via module theory, Graduate Texts in Mathematics, 136, Springer-Verlag, New York, 1992.
\bibitem{BE} A. Bertram and R. Easton, The tropical Nullstellensatz for congruences, Advanced in Mathematics, 308, pp. 36--82, 2017.
\bibitem{CDMY} D. Cartwright, A. Dudzik, M. Manjunath and Y. Yao, Embeddings and immersions of tropical curves, Collectanea mathematica, 67(1), pp. 1--19, 2016.
\bibitem{GG} J. Giansiracusa and N. Giansiracusa, Equations of tropical varieties, Duke Mathematical Journal, 165(18), pp. 3379--3433, 2016.
\bibitem{GKM} A. Gathmann, M. Kerber and H. Markwig, Tropical fans and the moduli spaces of tropical curves, Composito Mathematica, 145(1), pp. 173--195, 2009.
\bibitem{MR} D. Maclagan and F. Rinc\'{o}n, Tropical ideals, Composito Mathematica, 154(3), pp. 640--670, 2018.
\end{thebibliography}
\end{document}